\documentclass[reqno,11pt]{amsart}
\usepackage{amsfonts}
\usepackage{amssymb}
\usepackage{amsthm}
\usepackage{amsmath}
\usepackage{mathtools}
\usepackage{graphicx}
\usepackage[mathscr]{euscript}
\usepackage{sidecap}

\renewcommand{\thesection}{\arabic{section}}	
\numberwithin{equation}{section}
\DeclareMathOperator{\Eig}{Eig}		
\DeclareMathOperator{\ev}{ev}			
\DeclareMathOperator{\ind}{ind}			
\DeclareMathOperator{\lip}{lip}			
\DeclareMathOperator{\rk}{rk}			
\newcommand{\I}{{\mathbb I}}
\newcommand{\N}{{\mathbb N}}
\newcommand{\R}{{\mathbb R}}
\newcommand{\Z}{{\mathbb Z}}
\newcommand{\sF}{{\mathscr F}}
\newcommand{\sG}{{\mathscr G}}
\newcommand{\sL}{{\mathscr L}}
\newcommand{\sM}{{\mathscr M}}

\newcommand{\sS}{{\mathscr S}}

\newcommand{\fB}{{\mathfrak B}}
\newcommand{\fH}{{\mathfrak H}}

\newcommand{\eC}{{\mathscr C}}

\newcommand{\eL}{{\mathscr L}}
\newcommand{\eW}{{\mathscr W}}
\newcommand{\vphi}{\varphi}
\newcommand{\eps}{\varepsilon}
\newcommand{\tm}{\times}
\renewcommand{\leq}{\leqslant}
\renewcommand{\top}[1]{\xrightarrow[#1]{p}}
\newcommand{\intoo}[1]{\left(#1\right)}				
\newcommand{\intcc}[1]{\left[#1\right]}				
\newcommand{\set}[1]{\left\{#1\right\}}				
\newcommand{\abs}[1]{\left|#1\right|}				
\newcommand{\norm} [1]{\left\|#1\right\|}			

\newcommand{\fall}{\quad\text{for all }}
\newcommand{\on}{\quad\text{on }}
\newcommand{\cref}[1]{Cor.~\ref{#1}}
\newcommand{\eref}[1]{Example~\ref{#1}}
\newcommand{\fref}[1]{Fig.~\ref{#1}}
\newcommand{\lref}[1]{Lemma~\ref{#1}}
\newcommand{\pref}[1]{Prop.~\ref{#1}}
\newcommand{\rref}[1]{Rem.~\ref{#1}}
\newcommand{\tref}[1]{Thm.~\ref{#1}}

\theoremstyle{plain}
\newtheorem{thm}{Theorem}[section]

\newtheorem{lem}[thm]{Lemma}
\newtheorem{prop}[thm]{Proposition}
\newtheorem*{hyp}{Hypothesis}
\theoremstyle{definition}

\newtheorem{exam}[thm]{Example}
\newtheorem{rem}[thm]{Remark}
\allowdisplaybreaks

\begin{document}
\title{Global Continuation of Homoclinic Solutions}
\author{Christian P\"otzsche}
\author{Robert Skiba}
\address{Institut f\"ur Mathematik, Alpen-Adria Universit\"at Klagenfurt\\
Universit\"atsstra{\ss}e 65--67, 9020 Klagenfurt, Austria\\
\texttt{christian.poetzsche@aau.at}}

\address{Wydzia\l{} Matematyki i Informatyki, Uniwersytet Miko\l{}aja Kopernika w Toruniu\\
ul.\ Chopina 12/18, 87-100 Toru{\'n}, Poland\\
\texttt{robo@mat.uni.torun.pl}}

\begin{abstract}
	When extending bifurcation theory of dynamical systems to nonautonomous problems, it is a central observation that hyperbolic equilibria persist as bounded entire solutions under small temporally varying perturbations. In this paper, we abandon the smallness assumption and aim to investigate the global structure of the entity of all such bounded entire solutions in the situation of nonautonomous difference equations. Our tools are global implicit function theorems based on an ambient degree theory for Fredholm operators due to Fitzpatrick, Pejsachowicz and Rabier. For this we yet have to restrict to so-called homoclinic solutions, whose limit is $0$ in both time directions.
\end{abstract}


\maketitle
\section{Introduction}
The classical local theory of (discrete) dynamical systems deals with the behavior of finite-dimensional autonomous difference equations
\begin{equation}
	x_{t+1}=g(x_t,\alpha)
	\label{eqn}
\end{equation}
near given reference solutions, which are typically fixed or periodic points. An elementary application of the implicit function theorem implies that such periodic solutions persist under variation of the parameter $\alpha$ in \eqref{eqn}, provided they are hyperbolic and $\alpha$ is independent of time. Hyperbolicity is a generic property and means that there are no Floquet multipliers of the linearization on the unit circle of the complex plane.

In real-world models, yet, the parameter $\alpha$ describes the influence of the environment on a system \eqref{eqn} and thus it is more realistic and even natural to allow fluctuations of $\alpha$ in $t$. This leads to nonautonomous equations
\begin{equation}
	x_{t+1}=g(x_t,\alpha_t)
	\label{eqn2}
\end{equation}
and requires an extension of the established textbook theory (cf.~\cite{kloeden:rasmussen:10}), since aperiodic time-variant problems typically do not possess equilibria or periodic solutions. Already on this basic level one is confronted with the question to find adequate substitutes for equilibria under temporal forcing?

An answer can be given when \eqref{eqn} possesses a hyperbolic fixed point $\phi^\ast$ at a reference parameter value $\alpha^\ast$. Here, $\phi^\ast$ persists as a continuous branch $\alpha\mapsto\phi(\alpha)$ of bounded entire solution to \eqref{eqn2} with $\phi(\alpha^\ast)=\phi^\ast$ (typically not fixed points), as long as the parameter sequence $\alpha_t$, $t\in\Z$, remains uniformly close to $\alpha^\ast$ (cf.~\cite{poetzsche:08b}). The proof of this persistence result is again based on the implicit function theorem, but now applied to an operator equation between suitable sequence spaces. The condition yielding invertibility of the derivative is precisely an exponential dichotomy, which therefore represents the correct nonautonomous hyperbolicity concept. For general time-dependencies, however, an exponential dichotomy is not generic anymore.

While this approach yields information in the vicinity of a parameter~$\alpha^\ast$, it is nonetheless interesting to achieve insight on the global structure of the solution branch $\phi(\alpha)$. For this two approaches are conceivable:
\begin{itemize}
	\item[(1)] One works with analytical results guaranteeing (unique) existence over the whole parameter range (cf., for instance, \cite{radulescu:89}), which are in the spirit of the Hadamard-Levy theorem on global invertibility.

	\item[(2)] One applies a global implicit function theorem obtained from topological tools like a mapping degree.
\end{itemize}
In comparison, approach~(2) works under significantly weaker and for this reason interesting situations, if a feasible topological degree theory is available. Inspired by the works of \cite{morris:05,evequoz:09} or \cite{pejsachowicz:skiba:12, pejsachowicz:skiba:13} we employ a Fredholm degree developed in \cite{fitz-pejs-rabier:92,pejs-rabier:98}. However, since it relies on mappings having a constant Fredholm index $0$, this theory unfortunately does not apply to general bounded perturbations $(\alpha_t)_{t\in\Z}$. The resulting global implicit function Thms.~\ref{thmkiel} and \ref{thmevequoz} only apply to nonlinear Fredholm mappings of index~$0$. For bounded perturbations this can be guaranteed only locally. Dealing with solutions decaying to $0$, however, allows the argument that the Fredholm index is invariant under compact perturbations. In conclusion, we rather have to restrict to parameter sequences which asymptotically vanish in both time directions. Hence, we look for so-called homoclinic solutions and their global structure under variation of $\alpha$.
\subsection{Results and structure}
We are interested in the global structure of branches $C$ of homoclinic solutions emanating from a hyperbolic fixed point, or more general, from a hyperbolic bounded entire solution $\phi^\ast$, when varying the parameter $\lambda$ not only near some reference parameter $\lambda^\ast$, but over its whole range. We illustrate this by means of nonautonomous finite-dimensional difference equations
\begin{equation}
	\tag{$\Delta_\lambda$}
	\boxed{x_{t+1}=f_t(x_t,\lambda)}
\end{equation}
and roughly establish the following:
\begin{itemize}
	\item For right-hand sides of \eqref{deq} defined on a proper subset of $\R^d\tm\R$ the branches run from boundary to boundary, unless $C\setminus\set{(\phi^\ast,\lambda^\ast)}$ is connected (alternatives (a) and (b) of \tref{thmglobdeq}).

	\item If the right-hand sides are globally defined on $\R^d\tm\R$, then $C\setminus\set{(\phi^\ast,\lambda^\ast)}$ is either connected, or consists of two disjoint and unbounded branches (alternatives (c) and (d) of \tref{thmglobdeq}).
\end{itemize}
This classification of solution branches in \tref{thmglobdeq} is based on abstract results taken from \cite{kielhoefer:12,evequoz:09}. Up to our knowledge we present their first application to discrete time dynamical systems. Thereto, \eqref{deq} is understood as a parameter-dependent equation in the space of sequences with two-sided limit~$0$. Its analysis is based on preparations given in Sect.~\ref{sec2} and \ref{sec3}. Yet concepts and notions from dynamical systems are ubiquitous: In Sect.~\ref{sec4} we illustrate that the required Fredholm properties are closely connected to exponential dichotomies over the entire time axis $\Z$, as well as both half axes. Furthermore, a sufficient condition for properness is formulated in terms of limit sets for the Bebutov flow. Our result significantly extends the properness criterion from \cite{pejsachowicz:skiba:12}. These assumptions are particularly easy to verify in case of asymptotically periodic equations (see Sect.~\ref{sec52}). We close with various examples illustrating the main result. For the convenience of the reader, we conclude the paper with three appendices on our abstract global continuation results, the Bebutov flow/hull construction and finally sufficient criteria for unique bounded solutions.

Concerning related work, the global behavior of bifurcating solution branches in $\ell^2$ was studied in \cite{pejsachowicz:skiba:12}. Moreover, global continuation of solutions to boundary value problems for nonautonomous ordinary differential equations on the nonnegative halfline was considered in the inspiring references \cite{evequoz:09,morris:05}.
\subsection{Notation and sequence spaces}
A discrete interval $\I$ is the intersection of a real interval with the integers and $\I':=\set{t\in\I:\,t+1\in\I}$. We set $\Z_0^+:=\set{t\in\Z:\,t\geq 0}$, $\Z_0^-:=\set{t\in\Z:\,t\leq 0}$ for the half axes.

For Banach spaces $X,Y$ we denote the space of linear bounded operators between $X$ and $Y$ by $L(X,Y)$, $GL(X,Y)$ are the invertible elements and $F_0(X,Y)\subseteq L(X,Y)$ the Fredholm operators with index $0$. We briefly write $L(X):=L(X,X)$ (similarly for the other spaces) and $I_X$ for the identity mapping on $X$. Furthermore, $N(T):=T^{-1}(\set{0})$ and $R(T):=TX$ are the \emph{kernel} resp.\ the \emph{range} of an operator $T\in L(X,Y)$.

The cartesian product $X\tm Y$ is equipped with the norm
$$
	\norm{(x,y)}_{X\tm Y}:=\max\set{\norm{x}_X,\norm{y}_Y}
$$
throughout, and we write $\abs{\cdot}$ for a fixed norm on $\R^d$. Given a subset $O\subseteq X$, $\overline{O}$ denotes its closure. When $Z$ is a metric space and $\fB$ stands for a family of subsets of $Z$, a continuous $f:X\to Z$ is called \emph{proper on $\fB$}, if the preimages $f^{-1}(B)$ are compact for every $B\in\fB$.

Let $\ell^\infty(\Omega)$ be the set of bounded sequences $\phi=(\phi_t)_{t\in\Z}$ with values in $\Omega$ and $\ell^\infty:=\ell^\infty(\R^d)$ the Banach space of bounded sequences in $\R^d$ with norm
$$
	\norm{\phi}:=\sup_{t\in\Z}\abs{\phi_t}.
$$
The set $\ell_0$ of sequences with two-sided limit $0$ is a closed subspace of $\ell^\infty$. Convexity of $\Omega$ carries over to $\ell_0(\Omega)$ and so does openness. A sequence $(\phi^n)_{n\in\N}$ in $\ell^\infty$ is said to \emph{converge pointwise} to $\phi\in\ell^\infty$, if
$$
	\lim_{n\to\infty}\phi_t^n=\phi_t\fall t\in\Z
$$
holds and we abbreviate $\phi^n\top{n\to\infty}\phi$ in this case.

We introduce two bounded linear operators, namely the \emph{left shift}
\begin{align*}
	\sS\in L(\ell_0),\;\; (\sS\phi)_t:=\phi_{t+1}\fall t\in\Z
\end{align*}
and the \emph{evaluation operator}
\begin{align*}
	\ev_t\in L(\ell_0,\R^d),\;\;\ev_t\phi:=\phi_t\fall t\in\Z.
\end{align*}
The iterates of $\sS$ are denoted by $\sS^l$, $l\in\Z_0^+$. Notice that the shift $\sS$ is invertible with $(\sS^{-1}\phi)_t=\phi_{t-1}$ and therefore $\sS^l$ makes sense for all powers $l\in\Z$.

Let us next prepare compactness criteria in $\ell_0$, which are used to verify properness of nonlinear operators. We say a sequence $(\phi^n)_{n\in\N}$ in $\ell_0$ vanishes \emph{shiftly at} $\infty$, if for any increasing sequence $(k_n)_{n\in\N}$ in $\N$ and any sequence $(s_n)_{n\in\N}$ in $\Z$ with $\lim\limits_{n\to\infty}\abs{s_n}=\infty$, $\sS^{s_n}\phi^{k_n}\top{n\to\infty}\psi\in \ell^{\infty}$ it follows that $\psi=0$.
\begin{rem}\label{remtycho}
	(1) Note that pointwise convergence in $\ell^\infty$ does not imply weak convergence or boundedness. In order to illustrate this, we choose $d=1$ and write $\phi=(\ldots,\phi_{-1},\hat\phi_0,\phi_1,\ldots)$, i.e.\ mark the index $0$ element $\phi_0$ of $\phi$ with a hat. For example, let us take a sequence
	$$
		\phi^n:=(\ldots,0,\underbrace{\hat 1,\ldots,1}_{n\text{ times}},n,0,\ldots)\in\ell_0\fall n\in\N
	$$
	with pointwise limit $(\ldots,0,\hat 1,1,\ldots)$. Nevertheless, $(\phi^n)_{n\in\N}$ is not weakly convergent and of course unbounded due to $\norm{\phi^n}=n$ for all $n\in\N$.

	(2) From the sequential Tychonoff theorem it follows that, if a sequence $(\phi^n)_{n\in\N}$ in $\ell^\infty$ is bounded, then there exists a subsequence $(\phi^{n_k})_{k\in\N}$ such that $\phi^{n_k}\top{k\to\infty}\phi\in\ell^\infty$ (see \cite[p.~119, Prop.~1.8.12]{Tao:2010}).
\end{rem}

This brings us to the desired compactness characterization in $\ell_0$:
\begin{lem}[compactness in $\ell_0$]\label{lemcomp}
	For bounded $B\subset\ell_0$ are equivalent:
	\begin{enumerate}
		\item $B$ is relatively compact

		\item there exists a $\beta\in\ell_0(\R)$ such that $\abs{\phi_t}\leq\beta_t$ for all $t\in\Z$ and $\phi\in B$

		\item for sequences $(\phi^n)_{n\in\N}$ in $B$ and $(s_n)_{n\in\N}$ in $\Z$ with $\lim\limits_{n\to\infty}\abs{s_n}=\infty$ satisfying $\sS^{s_n}\phi^n\top{n\to\infty}\psi\in \ell^{\infty}$ it follows that $\psi=0$.
	\end{enumerate}
\end{lem}
\begin{proof}
	In \cite[Thm.~3]{banas:martinon:92} it is shown that the Hausdorff measure of noncompactness on $\ell_0$ is given by $\chi(B):=\lim_{n\to\infty}\sup_{\phi\in B}\sup_{n<|t|}\abs{\phi_t}$ and evidently $B\subset\ell_0$ is relatively compact, if and only if $\chi(B)=0$ holds.

	$(a)\Rightarrow(c)$: Let $(\phi^n)_{n\in\N}$ be a sequence in a relatively compact set $B\subset\ell_0$ and $(s_n)_{n\in\N}$, $\psi\in\ell^\infty$ be as in the above assertion. As $B$ is relatively compact, it follows that there exist $\phi\in\ell_0$ and a subsequence $(\phi^{n_k})_{k\in\N}$ such that
	\begin{equation}
		\label{lemcomp1}
		\lim_{k\to\infty}\norm{\sS^{s_{n_k}}\phi^{n_k}-\sS^{s_{n_k}}\phi}
		=
		\lim_{k\to\infty}\norm{\phi^{n_k}-\phi}
		=0\fall k\in\N,
	\end{equation}
	since the norm on $\ell_0$ is invariant under translations ($\sS$ is an isometry). As
	\begin{equation*}
		\sS^{s_{n_k}}\phi^{n_k}\top{k\to\infty}\psi
		\quad\text{ and }\quad
		\sS^{s_{n_k}}\phi\top{k\to\infty} 0,
	\end{equation*}
	it consequently results from \eqref{lemcomp1} that $\psi=0$.

	$(c)\Rightarrow(b)$: It suffices to show that
	$$
		\widetilde{\beta}_n:=\sup_{\psi^n\in B}\abs{\psi^n}\xrightarrow[n\to\infty]{}0.
	$$
	By contradiction, assume $(\widetilde{\beta}_n)_{n\in\N}$ does not converge to~$0$. Then there exist $\varepsilon>0$, a sequence $(\phi^n)_{n\in\N}$ in $B$ and a sequence of integers $(t_n)_{n\in\N}$ such that
	\begin{align*}
		\lim_{n\to\infty}|t_n|=\infty \text{ and } |\phi_{t_n}^n|\geqslant \eps\fall n\in\N.
	\end{align*}
	Now observe $\phi_{t_n}^n=\ev_0\sS^{t_n}\phi^n$. As $\sS^{t_n}\phi^n$ is bounded in the space $\ell^\infty$, we may assume w.l.o.g., in view of \rref{remtycho}(2), $\sS^{t_n}\phi^n\top{n\to\infty} \psi$ holds for some $\psi\in\ell^\infty$. Hence, it follows that $\psi=0$ and, in particular, $\lim_{n\to\infty}\ev_0\sS^{t_n}\phi^n=0$. This contradicts the fact that $\abs{\phi_{t_n}^n}\geqslant \eps$ for all $n\in\N$.

	$(b)\Rightarrow(a)$: Assume that there is $\beta\in\ell_0(\R)$ so that $\abs{\phi_t}\leq\beta_t$ for all $t\in\Z$ and $\phi\in B$. Then one infers $\chi(B)=0$ from
	\begin{equation*}
		\sup_{\phi\in B}\sup_{n<|t|}\abs{\phi_t}\leq \sup_{n<|t|}\beta_t\xrightarrow[n\to\infty]{}0
	\end{equation*}
	and the proof is complete.
\end{proof}
\section{Nonautonomous difference equations}
\label{sec2}
This paper addresses nonautonomous difference equations
\begin{equation}
	\tag{$\Delta_\lambda$}
	\boxed{x_{t+1}=f_t(x_t,\lambda),}
	\label{deq}
\end{equation}
whose right-hand side $f_t\colon\Omega\tm\Lambda\to\R^d$, $t\in\Z$, is defined on an open, convex neighborhood $\Omega\subseteq\R^d$ of $0$ and depends on a parameter $\lambda$. The \emph{general solution} to \eqref{deq} is given by
$$
	\vphi_\lambda(t;\tau,\xi)
	:=
	\begin{cases}
		\xi,&t=\tau,\\
		f_{t-1}(\cdot,\lambda)\circ\cdots\circ f_\tau(\cdot,\lambda)(\xi),&\tau<t,
	\end{cases}
$$
as long as the compositions stay in $\Omega$. An \emph{entire solution} to \eqref{deq} is a sequence $(\phi_t)_{t\in\Z}$ in $\Omega$ with $\phi_{t+1}\equiv f_t(\phi_t,\lambda)$ on $\Z$. For a fixed $\lambda^\ast\in\Lambda$ it is assumed throughout that there exists an entire solution $\phi^\ast$ to $(\Delta_{\lambda^\ast})$ satisfying
$$
	\lim\limits_{t\to\pm\infty}\phi_t^\ast=0.
$$
Such sequences are denoted as \emph{homoclinic} solutions with the trivial solution as immediate example.

In the following, we study the global structure of the set of homoclinic solutions to \eqref{deq} containing the pair $(\phi^\ast,\lambda^\ast)$ when $\lambda$ varies over the complete parameter space $\Lambda$. Our corresponding results based on functional analytical tools rely on two pillars, namely the Fredholmness and the properness of certain nonlinear operators, which we are going to prepare in the subsequent section. Throughout this requires to impose the standing
\begin{hyp}
	Let $\Lambda$ be an open interval, $\Omega\subseteq\R^d$ an open, convex neighborhood of $0$ and $\phi^\ast$ a homoclinic solution to $(\Delta_{\lambda^\ast})$ for some $\lambda^\ast\in\Lambda$. Assume that the continuous mappings $f_t\colon\Omega\tm\Lambda\to\R^d$, $t\in\Z$, satisfy:
	\begin{itemize}
		\item[$(H_0)$] For every compact $K\subset\R\tm\R^d$ one has
		$$
			\sup_{t\in\Z}\sup_{x\in K\cap(\Omega\tm\Lambda)}
			\abs{f_t(x,\lambda)}<\infty
		$$

		\item[$(H_1)$] for every $\eps>0$ and compact $K\subset\R\tm\R^d$ there exists a $\delta>0$ such that
		$$
			\max\set{\abs{x_2-x_1},\abs{\lambda_2-\lambda_1}}<\delta
			\quad\Rightarrow\quad
			\sup_{t\in\Z}\abs{f_t(x_2,\lambda_2)-f_t(x_1,\lambda_1)}<\eps
		$$
		for all $(x_1,\lambda_1),(x_2,\lambda_2)\in K\cap(\Omega\tm\Lambda)$

		\item[$(H_2)$] $D_1f_t\colon\Omega\tm\Lambda\to L(\R^d)$ exists as continuous function, for every bounded $B\subseteq\Omega$ one has
		$$
			\sup_{t\in\Z}\sup_{x\in B}
			\abs{D_1f_t(x,\lambda)}<\infty\fall\lambda\in\Lambda
		$$
		and for every $\eps>0$, $\lambda_0\in\Lambda$ there exists a $\delta>0$ such that
		$$
			\abs{x_2-x_1}<\delta
			\quad\Rightarrow\quad
			\sup_{t\in\Z}\abs{D_1f_t(x_2,\lambda)-D_1f_t(x_1,\lambda_0)}<\eps
		$$
		for all $x_1,x_2\in\Omega$, $\lambda\in B_\delta(\lambda_0)$

		\item[$(H_3)$] $\lim_{t\to\pm\infty}f_t(0,\lambda)=0$ for all $\lambda\in\Lambda$.
	\end{itemize}
\end{hyp}

Our preliminaries concerning the linear theory are as follows: For coefficients $A_t\in L(\R^d)$, $t\in\Z$, we consider a linear difference equation
\begin{equation}
	\tag{$L_0$}
	\boxed{x_{t+1}=A_tx_t}
	\label{deqlin}
\end{equation}
in $\R^d$ with the \emph{evolution operator} $\Phi_A:\set{(t,s)\in\Z^2\mid s\leq t}\to L(\R^d)$,
$$
	\Phi_A(t,s)
	:=
	\begin{cases}
		A_{t-1}\cdots A_s,&s<t,\\
		I_{\R^d},&s=t.
	\end{cases}
$$
Let $\I$ be an unbounded discrete interval. An \emph{invariant projector} is a sequence of projections $P_t\in L(\R^d)$, $t\in\I$, with
\begin{align*}
	A_{t+1}P_t=P_tA_t,\quad\quad
	A_t|_{N(P_t)}:N(P_t)\to N(P_{t+1})\text{ is invertible for all }t\in\I'.
\end{align*}
Hence, the restriction $\bar\Phi_A(t,s):=\Phi_A(t,s)|_{N(P_s)}\in GL(N(P_s),N(P_t))$ is well-defined for arbitrary $t,s\in\I$. One says the equation \eqref{deqlin} has an \emph{exponential dichotomy} (ED for short) on $\I$ with invariant projector $(P_t)_{t\in\I}$, if there exist reals $K\geq 1$, $\alpha\in(0,1)$ such that the exponential estimates
\begin{align}
	\abs{\Phi_A(t,s)P_s}\leq K\alpha^{t-s}, \quad\quad
	\abs{\bar\Phi_A(s,t)[I_{\R^d}-P_t]}&\leq K\alpha^{t-s}
	\fall s\leq t
	\label{edest}
\end{align}
and $s,t\in\I$ hold. The associate \emph{dichotomy spectrum} is given by
$$
	\Sigma_\I(A):=\set{\gamma>0\mid x_{t+1}=\gamma^{-1}A_tx_t\text{ has no ED on }\I}.
$$
In general, $\Sigma_\I(A)\subseteq(0,\infty)$ is the union of up to $d$ (closed) spectral intervals (for this, cf.~\cite[Thm.~4]{aulbach:minh:96}), which degenerate to points e.g.\ in the setting of
\begin{exam}[periodic linear equations]\label{exper}
	Let $p\in\N$. In case \eqref{deqlin} is $p$-periodic, i.e.\ $A_{t+p}=A_t$ holds for all $t\in\Z$, then $\Sigma_\I(A)=\sqrt[p]{\abs{\sigma(\Phi_A(p,0))}}\setminus\set{0}$ is discrete. In particular, for autonomous equations \eqref{deqlin} the dichotomy spectrum is given by the moduli of the nonzero eigenvalues $\Sigma_\I(A)=\abs{\sigma(A)}\setminus\set{0}$.
\end{exam}

It proves advantageous to introduce
\begin{align*}
	\Sigma^+(A)&:=\Sigma_{\Z_0^+}(A),&
	\Sigma^-(A)&:=\Sigma_{\Z_0^-}(A),&
	\Sigma(A)&:=\Sigma_{\Z}(A)
\end{align*}
as \emph{forward}, \emph{backward} resp.\ \emph{all time spectrum} of \eqref{deqlin}; it is $\Sigma^\pm(A)\subseteq\Sigma(A)$.

On the sequence space $\ell_0$ and for a bounded sequence $(A_t)_{t\in\Z}$ we introduce the bounded operator
\begin{align*}
	\eL_A\in L(\ell_0),\quad\quad
	(\eL_A\phi)_t:=\phi_t-A_{t-1}\phi_{t-1}\fall t\in\Z,
\end{align*}
whose Fredholm properties are as follows:
\begin{lem}[Fredholmness of $\eL_A$]\label{lempalmer}
	The following statements are equivalent:
	\begin{enumerate}
		\item $0\not\in\Sigma^+(A)$ and $0\not\in\Sigma^-(A)$ with corresponding projectors $P^+$ resp.\ $P^-$

		\item $\eL_A$ is Fredholm with $\ind\eL_A=\rk P_0^+-\rk P_0^-$.
	\end{enumerate}
\end{lem}
\begin{proof}
	$(a)\Rightarrow(b):$ See \cite[Thm.~8 and Cor.~17]{baskakov:00}.

	$(b)\Rightarrow(a):$ See \cite[Thm.~1.6]{latushkin:tomilov:05}.
\end{proof}
\section{Substitution operators on $\ell_0$}
\label{sec3}
Our overall approach is functional analytic and recursions \eqref{deq} are understood as abstract equations in ambient sequence spaces. This initially requires a careful analysis of the operators $\sF_0,\sG_0:\ell_0(\Omega)\to\ell_0$ defined by
\begin{align*}
	\sF_0(\phi)_t:=f_t(\phi_t),\quad\quad
	\sG_0(\phi)_t:=\phi_{t+1}-f_t(\phi_t)\fall t\in\Z,
\end{align*}
where the mappings $f_t\colon\Omega\to\R^d$, $t\in\Z$, are assumed to satisfy $(H_0)$-$(H_3)$ (without dependence on $\lambda$). As a result of \cite[Prop.~2.4, Thm.~2.5]{poetzsche:08b} both operators $\sF_0,\sG_0$ are well-defined.

At this point we remind the reader to some basic notions from topological dynamics (see App.~\ref{secB}, though). The \emph{hull} of a difference equation
\begin{equation}
	\tag{$\Delta$}
	\boxed{x_{t+1}=f_t(x_t)}
	\label{deq0}
\end{equation}
is denoted by $\fH(f)$ and equipped with the metric $\bar d$ given in \eqref{noddef}. Notice that in order to apply the results from App.~\ref{secB} one has to define $f(t,x):=f_t(x)$. From $(H_0)$ we see that $f$ is bounded, while $(H_1)$ yields the uniform continuity of $f$ on every compact $K\subset\R^d$. Hence, \lref{lemproplim} implies that both the hull $\fH(f)$, as well as the limit sets $\alpha(f),\omega(f)$ are nonempty compact sets.

Moreover, we say a subset $G\subseteq\fH(f)$ is \emph{admissible}, provided
\begin{itemize}
	\item $\set{\phi\in\ell^\infty\mid\phi_{t+1}\equiv g_t(\phi_t)\text{ on }\Z}=\set{0}$ for all $g\in G$.
\end{itemize}
This means that for every right-hand side $g_t:\Omega\to\R^d$, $t\in\Z$, the only bounded entire solution to $x_{t+1}=g_t(x_t)$ is the trivial one.

In what follows, we will need the next
\begin{lem}\label{lemconvweakII}
	If $(s_n)_{n\in\N}$ is a sequence of integers with $\lim_{n\to\infty}|s_n|=\infty$ and $\phi\in\ell_0$, then the sequence $(\phi^n)_{n\in\N}$ in $\ell_0$ pointwise given by
	$$
		\phi_t^n:=\phi_{t+s_n}\fall t\in\Z
	$$
	fulfills $\phi^n\top{n\to\infty}0$.
\end{lem}
\begin{proof}
	The implications
	\begin{align*}
		\phi\in\ell_0
		&\quad\Rightarrow\quad
		\phi_{t+s_n}\xrightarrow[n\to\infty]{}0\fall t\in\Z\\
		&\quad\Leftrightarrow\quad
		\phi_t^n\xrightarrow[n\to\infty]{}0\fall t\in\Z
		\quad\Leftrightarrow\quad
		\phi^n\top{n\to\infty}0
	\end{align*}
	guarantee the assertion.
\end{proof}
\begin{lem}[properness]\label{lemproper}
	If $\alpha(f),\omega(f)$ are admissible, then $\sG_0:\ell_0(\Omega)\to\ell_0$ is proper on all bounded, closed $B\subset\ell_0(\Omega)$.
\end{lem}
\begin{proof}
	Above all, note that in view of \lref{lemcomp} it suffices to show that any bounded sequence $(\phi^n)_{n\in\N}$ in $\ell_0(\Omega)$ satisfying
	$$
		\norm{\sG_0(\phi^n)-\varphi}\xrightarrow[n\to\infty]{} 0
		\text{ with some }\varphi\in\ell_0,
	$$
	vanishes shiftly at $\infty$. Take any increasing sequence $(k_n)_{n\in\N}$ in $\N$ and any sequence $(s_n)_{n\in\N}$ in $\Z$ satisfying $\lim\limits_{n\to\infty}\abs{s_n}=\infty$ such that
\begin{equation}\label{lemproper1}
	\sS^{s_n}\phi^{k_n}\top{n\to\infty}\psi\text{ with some }\psi\in\ell^{\infty}.
\end{equation}
We must show that $\psi=0$. For this purpose, observe
	\begin{equation}
		\norm{\sS^{s_n}\sG_0(\phi^{k_n})-\sS^{s_n}\varphi}
		=
		\norm{\sG_0(\phi^{k_n})-\varphi}
		\xrightarrow[n\to\infty]{}0
		\label{lemproper2}
	\end{equation}
	and put $f^n:=f(\cdot+s_n,\cdot)\in\fH(f)$. Because $\fH(f)$ is compact, we can deduce that there exists a subsequence $(s_{n_i})_{i\in\N}$ such that
	\begin{equation}
		s_{n_i}>0 \text{ and } \bar d(f^{n_i},f^+)\xrightarrow[i\to\infty]{}0 \text{ for some }f^+\in\omega(f)\subseteq\fH(f)
		\label{lemproper3}
\end{equation}
or
\begin{equation}
	s_{n_i}<0 \text{ and } \bar d(f^{n_i},f^-)\xrightarrow[i\to\infty]{}0 \text{ for some }f^-\in\alpha(f)\subseteq\fH(f).
		\label{lemproper3s}
\end{equation}
In case \eqref{lemproper3} we introduce the following limit operators
	\begin{align*}
		\sF^+:&\ell_0\to\ell_0,\quad
		\sF^+(\phi)_t:=f_t^+(\phi_t),\\
		\sG^+:&\ell_0\to\ell_0,\quad
		\sG^+(\phi)_t:=\phi_{t+1}-f_t^+(\phi_t).
	\end{align*}
Since $\omega(f)$ is admissible, it suffices to prove that $\sG^+(\psi)=0$ and we proceed as follows: First, \eqref{lemproper3} implies that
$$
	f(t+s_{n_i},\psi_t)\xrightarrow[i\to\infty]{}f_t^+(\psi_t)\fall t\in\Z
$$
and \eqref{lemproper1} with \eqref{uniformly} leads to
$$
	f(t+s_{m_i},\phi^{k_i}_{t+s_{n_i}})-f(t+s_{n_i},\psi_t)\xrightarrow[i\to\infty]{} 0
	\fall t\in\Z.
$$
Second, \eqref{lemproper2} leads to
$$
	\left(\phi_{t+1+s_i}^i-f\left(t+s_{n_i},\phi^{k_i}_{t+s_{n_i}}\right)\right)-\varphi_{t+s_{n_i}}\xrightarrow[i\to\infty]{} 0\fall t\in\Z,
$$
while \lref{lemconvweakII} and \eqref{lemproper1} guarantee
$$
	\varphi_{t+s_{n_i}} \xrightarrow[n\to\infty]{} 0
	\text{ and }
	\phi_{t+1+s_{n_i}}^{k_i} \xrightarrow[i\to\infty]{}\psi_{t+1}\fall t\in\Z.
$$
Finally from the above we deduce that
\begin{align*}
\phi_{t+1+s_{n_i}}^{k_i}
-f(t+s_{n_i},\psi_t) \xrightarrow[i\to\infty]{} 0
\end{align*}
and
\begin{align*}
\phi_{t+1+s_{n_i}}^{k_i}
-f(t+s_{n_i},\psi_t) \xrightarrow[i\to\infty]{} \psi_{t+1}-f_t^+(\psi_t).
\end{align*}
	Hence, we infer that $\sG_0(\psi)=0$. Since the dual case \eqref{lemproper3s} can be treated similarly, the admissibility of $\alpha(f)$ completes the proof.
\end{proof}

Our further analysis is based on the substitution operators
\begin{align}
	\sF(\phi,\lambda)_t:=f_t(\phi_t,\lambda),\quad\quad
	\sF^j(\phi,\lambda)_t:=D_1^jf_t(\phi_t,\lambda)
	\fall t\in\Z,
	\label{nofdef}
\end{align}
$\phi\in\ell_0(\Omega)$, $\lambda\in\Lambda$ and indices $j\in\set{0,1}$, whose properties are as follows:
\begin{prop}\label{propF}
	The operator $\sF:\ell_0(\Omega)\tm\Lambda\to\ell_0$ is well-defined with the following properties for every $\phi\in\ell_0(\Omega)$, $\lambda\in\Lambda$:
	\begin{enumerate}
		\item $\sF^j:\ell_0(\Omega)\tm\Lambda\to L_j(\ell_0)$ is continuous for $j\in\set{0,1}$

		\item $D_1\sF:\ell_0(\Omega)\tm\Lambda\to L(\ell_0)$ exists with $D_1\sF(\phi,\lambda)=\sF^1(\phi,\lambda)$

		\item $D_1\sF(\phi,\lambda)-D_1\sF(\phi^\ast,\lambda)\in L(\ell_0)$ is compact.
	\end{enumerate}
\end{prop}
\begin{proof}
	(a) and (b) were essentially shown in \cite[Prop.~2.4]{poetzsche:08b}, so is the well-definedness of $\sF$.

	(c): Since linear combinations of compact operators are compact (see \cite[p.~278, Thm.~(i)]{yosida:80}) it suffices to show that $D_1\sF(\phi,\lambda)-D_1\sF(\phi^\ast,\lambda)$ is compact for all $\phi\in\ell_0(\Omega)$ and $\lambda\in\Lambda$. Due to the representation (cf.~\eqref{nofdef})
	$$
		[(D_1\sF(\phi,\lambda)-D_1\sF(\phi^\ast,\lambda))\psi]_t
		=
		\underbrace{(D_1f_t(\phi_t,\lambda)-D_1f_t(\phi_t^\ast,\lambda))}_{=:A_t}\psi_t
		\fall t\in\Z
	$$
	and $\psi\in\ell_0$ we see that $D_1\sF(\phi,\lambda)-D_1\sF(\phi^\ast,\lambda)$ is a multiplication operator $\sM\in L(\ell_0)$, $(\sM\psi)_t:=A_t\psi_t$. To establish its compactness, let $\eps>0$. Thanks to $(H_1)$ there exists a $\delta>0$ such that $\abs{x}<\delta$ implies
	$$
		\abs{D_1f_t(x,\lambda)-D_1f_t(\phi_t^\ast,\lambda)}<\eps\fall t\in\Z.
	$$
	Hence, because of $\phi\in\ell_0(\Omega)$ we find a $T\in\Z$ with $\abs{\phi_t}<\delta$ and therefore
	$$
		\abs{A_t}=\abs{D_1f_t(\phi_t,\lambda)-D_1f_t(\phi_t^\ast,\lambda)}<\eps
		\fall T\leq\abs{t},
	$$
	which implies that $\lim_{t\to\pm\infty}A_t=0$. It remains to show that $\sM\in L(\ell_0)$ is compact. Thereto, consider the sequence of compact operators $\sM_k\in L(\ell_0)$,
	$$
		(\sM_k\psi)_t
		:=
		\begin{cases}
			A_t\psi_t,&t\in[-k,k]\cap\Z,\\
			0,&\text{else}
		\end{cases}
		\fall k\in\N
	$$
	satisfying \[[(\sM-\sM_k)\psi]_t=\sup_{k<\abs{t}}\abs{A_t\psi_t}\leq\sup_{k<\abs{t}}\abs{A_t}\norm{\psi}\] for all $t\in\Z$. This yields that $\lim\limits_{k\to\infty}\norm{\sM-\sM_k}=0$, $\sM$ is the uniform limit of a sequence of compact operators and thus compact (cf.~\cite[p.~278, Thm.~(iii)]{yosida:80}).
\end{proof}
\section{Entire hyperbolic solutions}
\label{sec4}
Let us consider the linear difference equation
\begin{equation}
	\tag{$V_\lambda$}
	\boxed{x_{t+1}=D_1f_t(\phi_t^\ast,\lambda)x_t,}
	\label{var}
\end{equation}
with dichotomy spectra denoted by $\Sigma(\lambda)$ and $\Sigma^-(\lambda),\Sigma^+(\lambda)$ for $\lambda\in\Lambda$. Since $\phi^\ast$ needs not to be a solution to \eqref{deq}, note that in general only $(V_{\lambda^\ast})$ is a variational equation.

In case $1\not\in\Sigma(\lambda^\ast)$ it follows from the usual local implicit function theorem that there is a neighborhood $\Lambda_0\subseteq\Lambda$ of $\lambda^\ast$ and a continuous function $\phi:\Lambda_0\to\ell_0$ (the local branch) such that $\phi(\lambda)$ is the unique homoclinic solution to \eqref{deq} (see \cite[Thm.~2.17]{poetzsche:08b}) in a neighborhood of $(\phi^\ast,\lambda^\ast)$. In the following, we are interested in the global structure of the component
$$
	C\subseteq\set{(\phi,\lambda)\in\ell_0(\Omega)\tm\Lambda\mid\phi_{t+1}\equiv f_t(\phi_t,\lambda)\text{ on }\Z}
$$
containing the pair $(\phi^\ast,\lambda^\ast)$. A continuation result for homoclinic solutions to \eqref{deq} relies on an immediate but crucial tool for our overall approach:
\begin{lem}\label{lemgB}
	Let $\lambda\in\Lambda$. A sequence $\phi\in\ell_0(\Omega)$ solves the difference equation \eqref{deq} if and only if $\phi$ satisfies the nonlinear operator equation
	\begin{equation}
		\tag{$O_\lambda$}
		\sG(\phi,\lambda)=0
		\label{opiB}
	\end{equation}
	with the operator $\sG:\ell_0(\Omega)\tm\Lambda\to\ell_0$ given by $\sG(\phi,\lambda):=\sS\phi-\sF(\phi,\lambda)$.
\end{lem}
\begin{proof}
	The well-definedness of $\sG$ immediately follows from \pref{propF}. The equivalence statement is clear.
\end{proof}
By means of \pref{propF} our assumptions imply that the partial derivative
$$
	D_1\sG:\ell_0(\Omega)\tm\Lambda\to L(\ell_0)
$$
exists as a continuous function of the form
$$
	D_1\sG(\phi,\lambda)\psi
	=
	\sS\psi-D_1\sF(\phi,\lambda)\psi\fall\psi\in\ell_0
$$
and possesses the following properties:
\begin{lem}\label{lemadmiss}
	For all $\phi\in\ell_0(\Omega)$ and $\lambda\in\Lambda$ one has:
	\begin{enumerate}
		\item $D_1\sG(\phi^\ast,\lambda^\ast)\in GL(\ell_0)\Leftrightarrow 1\not\in\Sigma(\lambda^\ast)$

		\item $1\not\in\Sigma(\lambda^\ast)\Rightarrow D_1\sG(\phi,\lambda^\ast)\in F_0(\ell_0)$

		\item $D_1\sG(\phi^\ast,\lambda)\in F_0(\ell_0)\Leftrightarrow D_1\sG(\phi,\lambda)\in F_0(\ell_0)$
	\end{enumerate}
\end{lem}
\begin{proof}
	(a): For fixed $\lambda\in\Lambda$ this is a consequence of \cite[Thm.~2, Cor.~3]{aulbach:minh:96}.

	(b): Due to (a) one has $D_1\sG(\phi^\ast,\lambda^\ast)\in GL(\ell_0)$. This obviously guarantees $D_1\sG(\phi^\ast,\lambda^\ast)\in F_0(\ell_0)$ and combining
	\begin{align*}
		D_1\sG(\phi,\lambda^\ast)
		&=
		D_1\sG(\phi^\ast,\lambda^\ast)+D_1\sG(\phi,\lambda^\ast)-D_1\sG(\phi^\ast,\lambda^\ast)\\
		&=
		D_1\sG(\phi^\ast,\lambda^\ast)+D_1\sF(\phi,\lambda^\ast)-D_1\sF(\phi^\ast,\lambda^\ast)\fall\phi\in\ell_0(\Omega)
	\end{align*}
	with \pref{propF}(c) shows that $D_1\sG(\phi,\lambda^\ast)$ is a compact perturbation of a Fredholm operator with index $0$. Since compact perturbations do not affect the index (see \cite[p.~161, Thm.~16]{mueller:07}), we obtain $D_1\sG(\phi,\lambda^\ast)\in F_0(\ell_0)$.

	(c): The nontrivial implication assumes $D_1G(\phi^\ast,\lambda)\in F_0(\ell_0)$ and
	$$
		D_1\sG(\phi,\lambda)
		=
		D_1\sG(\phi^\ast,\lambda)+D_1\sF(\phi,\lambda)-D_1\sF(\phi^\ast,\lambda)
		\fall\phi\in\ell_0(\Omega),\,\lambda\in\Lambda
	$$
	represents $D_1\sG(\phi,\lambda)$ as compact perturbation of $D_1\sG(\phi^\ast,\lambda)$. The claim follows as in the proof of (b).
\end{proof}

While this already settles our required Fredholm theory, we continue with a general criterion for the properness for $\sG$. It is based on concepts from topological dynamics introduced in App.~\ref{secB}. In particular, slightly modifying the notation there, rather than $\alpha(f(\cdot,\lambda))$ and $\omega(f(\cdot,\lambda))$, in order to emphasize the parameter dependence, we write $\alpha(\lambda)$ resp.\ $\omega(\lambda)$ to denote the limit sets of the right-hand side to \eqref{deq} for $\lambda\in\Lambda$.
\begin{prop}[properness]\label{propropI}
	If $\alpha(\lambda),\omega(\lambda)$ are admissible for all $\lambda\in\Lambda$, then $\sG\colon\ell_0(\Omega)\tm\Lambda\to\ell_0$ is proper on every product $B\tm\Lambda_0$ with $B\subset\ell_0(\Omega)$ bounded, closed and $\Lambda_0\subseteq\Lambda$ compact.
\end{prop}
\begin{proof}
	By \pref{propF} the function $\sG:\ell_0(\Omega)\tm\Lambda\to\ell_0$ is continuous.
%
%
	Let $B\subset\ell_0(\Omega)$ be closed, bounded and suppose $\Lambda_0\subset\Lambda$ is compact. Then $\sG$ is proper on such $B\tm\Lambda_0$, if and only if, for all compact $K\subset\ell_0$ the set $\sG^{-1}(K)\cap (B\tm\Lambda_0)$ is compact. This is equivalent to the fact that for all such $K\subset\ell_0$, any sequence in $\sG^{-1}(K)\cap (B\tm\Lambda_0)$ admits a convergent subsequence. Thus, take a compact subset $K\subset\ell_0$ and let $((\phi^n,\lambda_n))_{n\in\N}$ be a sequence in $\sG^{-1}(K)\cap (B\tm\Lambda_0)$. Since $K$ is compact, there exists a $\psi\in\ell_0$ and a subsequence $((\phi^{n_{i}},\lambda_{n_{i}}))_{i\in\N}$ such that
	\begin{equation}\label{proprop1}
		\|\sG(\phi^{n_{i}},\lambda_{n_{i}})-\psi\|\xrightarrow[i\to\infty]{} 0.
	\end{equation}
	Because $\Lambda_0$ is compact, one finds a convergent subsequence $(\lambda_{n_{i_j}})_{j\in \N}$ with limit $\lambda_0\in\Lambda_0$. Using \lref{lemproper}, $\sG(\cdot,\lambda_0)$ is proper on the bounded, closed subsets
of $\ell_0$ and we are about to prove
	$$
		\norm{\sG(\phi^{n_{i_j}},\lambda_0)-\psi}
		\xrightarrow[j\to\infty]{} 0.
	$$
	Thereto, we have from the triangle inequality and \lref{lemgB} that
	\begin{align*}
		\norm{\sG(\phi^{n_{i_j}},\lambda_0)-\psi}
		&\leq
		\left\|\sG(\phi^{n_{i_j}},\lambda_0)-\sG(\phi^{n_{i_j}},\lambda_{n_{i_j}})\right\|
		+
		\left\|\sG(\phi^{n_{i_j}},\lambda_{n_{i_j}})-\psi\right\|\\
		&=
		\norm{\sF(\phi^{n_{i_j}},\lambda_0)-\sF(\phi^{n_{i_j}},\lambda_{n_{i_j}})}
		+
		\left\|\sG(\phi^{n_{i_j}},\lambda_{n_{i_j}})-\psi\right\|
	\end{align*}
	for all $j\in\N$ and with a view to \eqref{proprop1} it remains to establish
	\begin{equation}
		\norm{\sF(\phi^{n_{i_j}},\lambda_0)-\sF(\phi^{n_{i_j}},\lambda_{n_{i_j}})}
		\xrightarrow[j\to\infty]{} 0.
		\label{proprop3}
	\end{equation}
	Indeed, since $K\subset\ell_0$ is bounded, it follows that there exists $M>0$ such that $|\phi_t^{n_{i_j}}|\leq M$ for all $t\in\Z$ and $j\in\N$. Consequently, $(H_1)$ implies that
	\begin{equation*}
		\abs{f_t(\phi_t^{n_{i_j}},\lambda_{n_{i_j}})-f_t(\phi_t^{n_{i_j}},\lambda_0)}
		\xrightarrow[j\to\infty]{}
		0\quad\text{uniformly in }t\in\Z,
	\end{equation*}
	and \eqref{nofdef} leads to \eqref{proprop3}. Finally, since $\sG(\cdot,\lambda_0)$ is proper, it follows that also $(\phi^{n_{i_j}})_{j\in \N}$ has a convergent subsequence, which guarantees a convergent subsequence of $((\phi^{n_{i}},\lambda_{n_{i}}))_{i\in\N}$. This completes the proof.
\end{proof}

We arrive at our main result, which supplements the local continuation property of \cite[Thm.~2.17]{poetzsche:08b}, but requires a real parameter $\lambda$.
\begin{thm}[global continuation in $\ell_0$]\label{thmglobdeq}
	Beyond $(H_0)$-$(H_3)$ let us assume for all $\lambda\in\Lambda$:
	\begin{itemize}
		\item[(i)] The linear equations \eqref{var} satisfy
		\begin{align}
			1\not\in\Sigma(\lambda^\ast),\quad\quad
			1\not\in\Sigma^+(\lambda),\quad\quad
			1\not\in\Sigma^-(\lambda)
			\label{thmglobdeq2}
		\end{align}
		with corresponding invariant projectors such that $\rk P_0^+(\lambda)=\rk P_0^-(\lambda)$

		\item[(ii)] $\alpha(\lambda),\omega(\lambda)$ are admissible.
	\end{itemize}
	If $C\subseteq\ell_0(\Omega)\tm\Lambda$ denotes the component of homoclinic solutions to \eqref{deq} containing $(\phi^\ast,\lambda^\ast)$ and
	\begin{align*}
		C_-:=\set{(\phi,\lambda)\in C\mid\lambda\leq\lambda^\ast},\quad\quad
		C_+:=\set{(\phi,\lambda)\in C\mid \lambda^\ast\leq\lambda},
	\end{align*}
	then (at least) one the following alternatives applies (cf.~\fref{figglob3}):
	\begin{figure}[ht]
		\centering
		\includegraphics[scale=0.5]{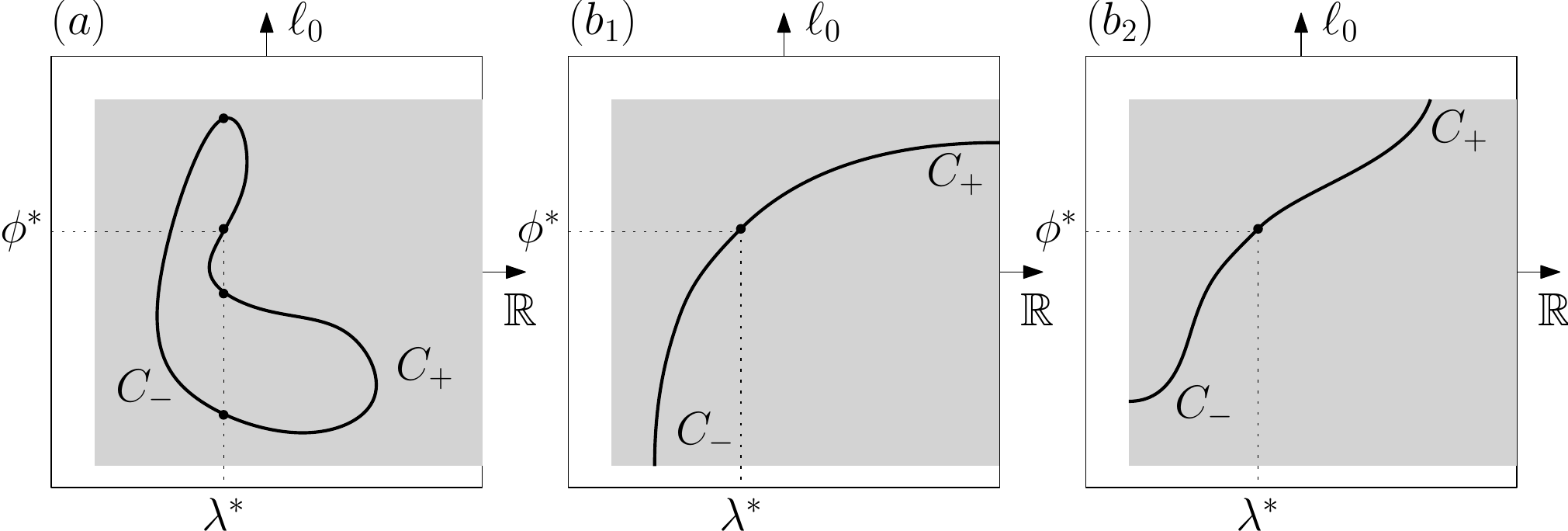}
		\caption{Alternatives from \tref{thmglobdeq}, where the grey shaded area symbolizes $\ell_0(\Omega)\tm\Lambda$: (a) The intersection $C_-\cap C_+$ is larger than just $\set{(\phi^\ast,\lambda^\ast)}$ or, $(b_1)$ $C_+$ is unbounded (here $C_-$ touches the boundary of $\ell_0(\Omega)$) or, $(b_2)$ $C_-$ touches the boundary of $\Lambda$ (while $C_+$ touches the boundary of $\ell_0(\Omega)$)}
		\label{figglob3}
	\end{figure}
	\begin{enumerate}
		\item $C_-\cap C_+\neq\set{(\phi^\ast,\lambda^\ast)}$
		
		\item the branches $C_+$ and $C_-$ are connected and
		\begin{enumerate}
			\item[$(b_1)$] $C_+$ is unbounded or at least one of the following sets is nonempty:
			$$
				\Pi_1(\overline{C_+})\cap\partial\ell_0(\Omega),\,
				\overline{\Pi_2(C_+)}\cap\partial\Lambda
			$$

			\item[$(b_2)$] $C_-$ is unbounded or at least one of the following sets is nonempty:
			$$
				\Pi_1(\overline{C_-})\cap\partial\ell_0(\Omega),\,
				\overline{\Pi_2(C_-)}\cap\partial\Lambda,
			$$
		\end{enumerate}
	\end{enumerate}
	where $\Pi_1,\Pi_2$ are the projection of $(x,\lambda)$ onto the first resp.\ second component. For $\Omega=\R^d$, $\Lambda=\R$ (exactly) one of the next cases occurs (cf.~\fref{figglob1}):
	\begin{figure}[ht]
		\centering
		\includegraphics[scale=0.5]{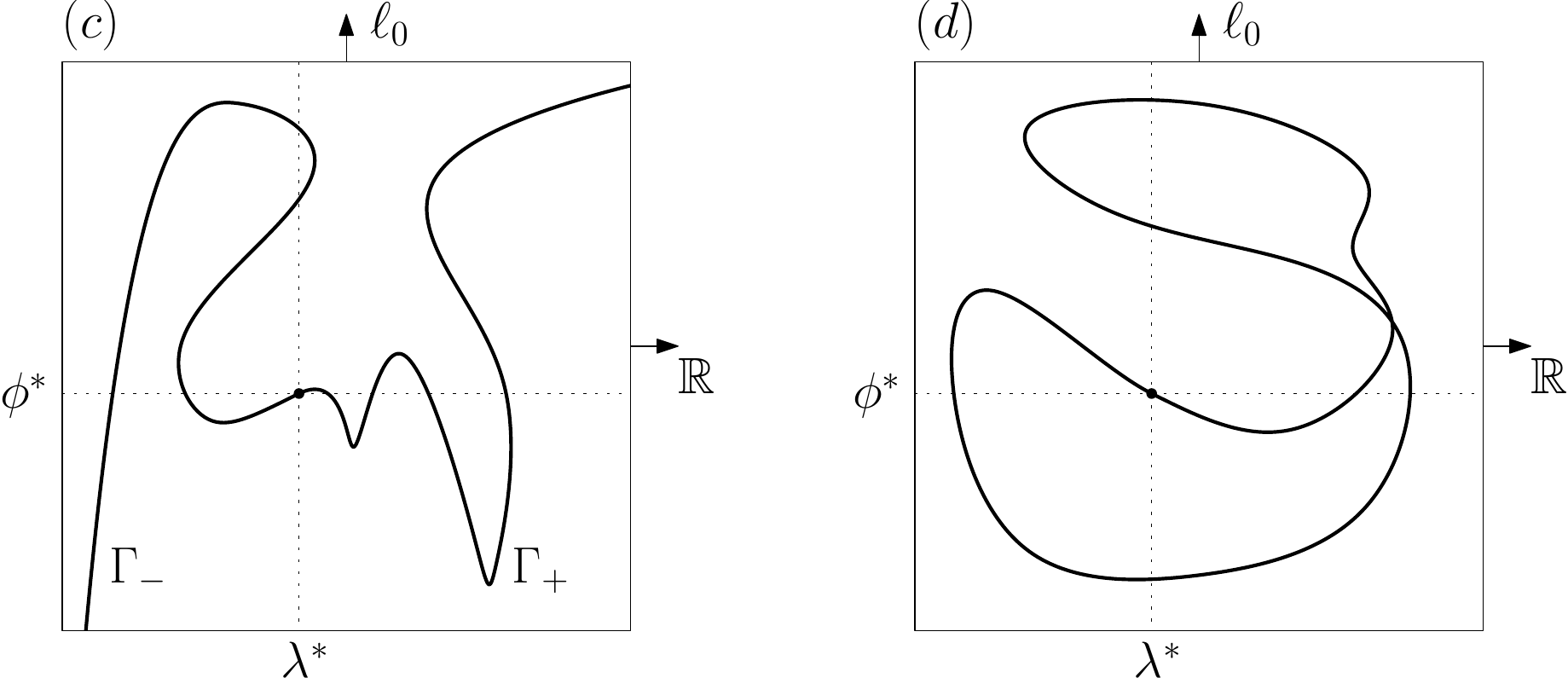}
		\caption{Alternatives from \tref{thmglobdeq}: (c) Two disjoint unbounded sets $\Gamma_-,\Gamma_+$ meet at $(\phi^\ast,\lambda^\ast)$ or, (d) the difference $C\setminus\set{(\phi^\ast,\lambda^\ast)}$ is connected}
		\label{figglob1}
	\end{figure}
	\begin{enumerate}
		\item[(c)] $C=\set{(\phi^\ast,\lambda^\ast)}\cup\Gamma_+\cup\Gamma_-$ with unbounded disjoint sets $\Gamma_-,\Gamma_+$,

		\item[(d)] $C\setminus\set{(\phi^\ast,\lambda^\ast)}$ is connected.
	\end{enumerate}
\end{thm}
\begin{rem}
	(1) Compared to the local result \cite[Thm.~2.17]{poetzsche:08b} preceding \tref{thmglobdeq}, we assume slightly weaker differentiability, but stronger continuity assumptions on the right-hand sides $f_t$. Thus, locally near $(\phi^\ast,\lambda^\ast)$ the component $C$ is in general merely graph of a continuous function over $\Lambda$, and no longer of class $C^1$.

	(2) The admissibility assumption (ii) can be verified using the criteria from App.~\ref{appC} for the unique existence of bounded entire solutions to \eqref{deq}.

	(3) Note that \tref{thmglobdeq} applies to hyperbolic fixed points $x^\ast=g(x^\ast,\alpha^\ast)$ of \eqref{eqn} under time-dependent forcing of the form $\alpha_t=\alpha^\ast+\lambda\mu_t$, where $(\mu_t)_{t\in\Z}$ is decaying to $0$ and $\lambda\in\R$ controls the magnitude of the perturbation. For this, consider the trivial solution $\phi^\ast=0$ of \eqref{deq} with the right-hand side
	$$
		f_t(x,\lambda):=g(x+x^\ast,\alpha^\ast+\lambda\mu_t)-g(x^\ast,\alpha^\ast+\lambda\mu_t)
	$$
	and the parameter value $\lambda^\ast=0$. This idea extends to periodic (and more general) hyperbolic solutions to \eqref{eqn}.
\end{rem}
\begin{proof}
	Because the openness of $\Omega$ extends to $\ell_0(\Omega)$, we can apply the abstract Thms.~\ref{thmkiel} and \ref{thmevequoz} to $\sG:\overline{O}\tm\Lambda\to\ell_0$ from \lref{lemgB} with $O:=\ell_0(\Omega)$. Since $\sS$ is a bounded linear operator, it results from \pref{propF}(a) that $\sG$ is continuous. Moreover, due to \pref{propF}(b) the derivative $D_1\sF:O\tm\Lambda\to L(\ell_0)$ exists as a continuous function and it results that also $D_1\sG$ exists with
	$$
		D_1\sG(\phi^\ast,\lambda^\ast)\psi=\sS\psi-D_1\sF(\phi^\ast,\lambda^\ast)\psi\fall\psi\in\ell_0.
	$$
	\underline{ad \eqref{glob1}}: Thanks to \lref{lemgB} it is clear that $\sG(\phi^\ast,\lambda^\ast)=0$ holds.
	\\
	\underline{ad \eqref{glob2}}: Because of the first inclusion in \eqref{thmglobdeq2} the derivative $D_1\sG(\phi^\ast,\lambda^\ast)$ is invertible due to \lref{lemadmiss}(a).
	\\
	\underline{ad \eqref{glob3}}: Let $\lambda\in\Lambda$. The remaining inclusions of \eqref{thmglobdeq2} guarantee that \eqref{var} has EDs on both $\Z_0^+$ and $\Z_0^-$. The assumptions on the corresponding projectors thus imply $D_1\sG(\phi^\ast,\lambda)\in F_0(\ell_0)$ due to \lref{lempalmer}. Finally, \lref{lemadmiss}(c) ensures that also $D_1\sG(\phi,\lambda)$ is Fredholm of index $0$ for all $\phi\in\ell_0$.
	\\
	\underline{ad \eqref{glob4}}: We derive from \pref{propropI} that $\sG$ is proper on every $B\tm\Lambda_0$ with bounded, closed $B\subset\ell_0(\Omega)$ and compact $\Lambda_0\subseteq\Lambda$.
	
	Now the assertions (a), (b) result from \tref{thmevequoz}, while \tref{thmkiel} applied to \eqref{opiB} ensures the two alternatives (c), (d).
\end{proof}
\section{Applications}
In this section, we collect several types of difference equations with properties in the scope our main \tref{thmglobdeq}.
\begin{figure}[ht]
	\centering
	\includegraphics[scale=0.5]{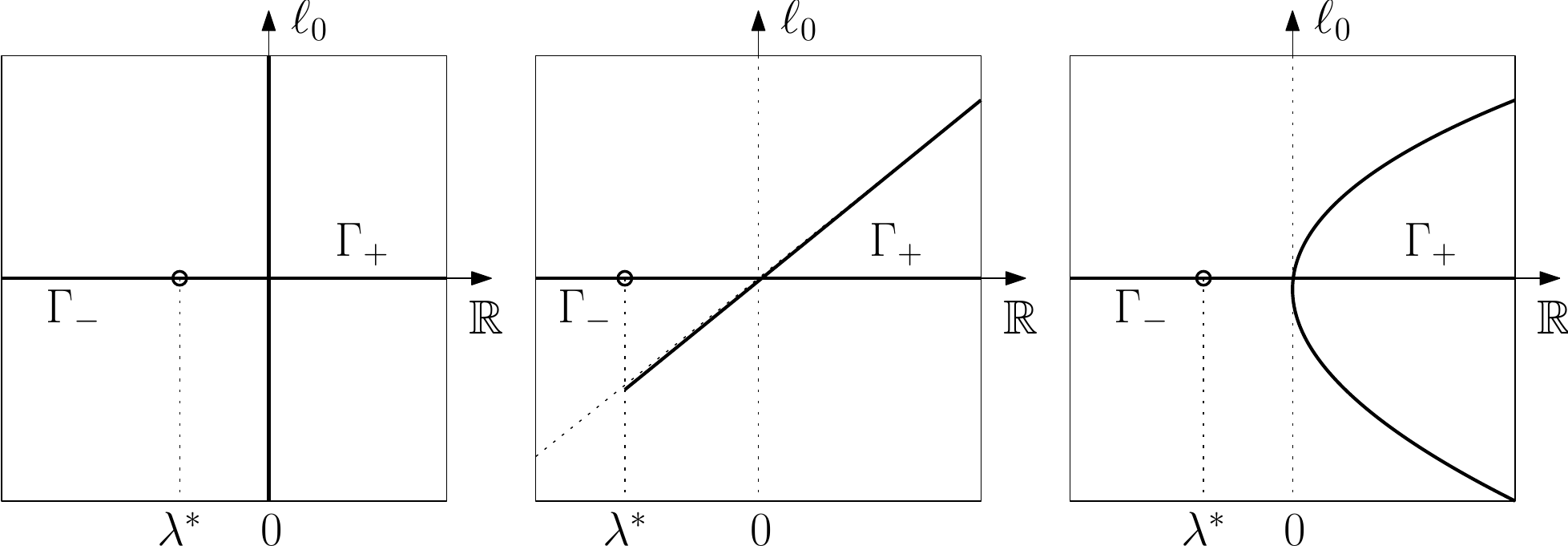}
	\caption{The schematic sets $G^{-1}(0)\subseteq\ell_0(\R^2)\tm\R$ of homoclinic solutions to \eqref{deq} and $\Gamma_-,\Gamma_+$ for the examples from Sect.~\ref{sec50}}
	\label{figcon}
\end{figure}
\subsection{Piecewise constant equations}
\label{sec50}
Let us first illuminate \tref{thmglobdeq} in the light of concrete examples from the bifurcation theory of \cite{poetzsche:08c}. They allow to determine the set of all homoclinic solutions, and particularly the branch $C$ explicitly.

Suppose that $\alpha\in(-1,1)$ is a fixed nonzero real and $\lambda\in\R$ serves as continuation parameter. We consider the linear homogeneous equation
\begin{equation}
	x_{t+1}
	=
	f_t(x_t,\lambda)
	:=
	\begin{pmatrix}
		b_t& 0\\
		\lambda&c_t
	\end{pmatrix}x_t
	\label{exA01}
\end{equation}
with asymptotically constant sequences
\begin{align*}
	b_t:=
	\begin{cases}
		\alpha^{-1},& t<0,\\
		\alpha,& t\geq 0,
	\end{cases}\quad\quad
	c_t:=
	\begin{cases}
		\alpha,& t<0,\\
		\alpha^{-1},& t\geq 0.
	\end{cases}
\end{align*}
On the one hand, since \eqref{exA01} is triangular, the dichotomy spectra reads as
\begin{align}
	\Sigma(\lambda)
	&=
	\begin{cases}
		[\alpha,\tfrac{1}{\alpha}],&\lambda=0,\\
		\set{\alpha,\tfrac{1}{\alpha}},&\lambda\neq 0,
	\end{cases}\quad\quad
	\Sigma^\pm(\lambda)
	=
	\set{\alpha,\tfrac{1}{\alpha}}
	\fall\lambda\in\R.
	\label{prolin}
\end{align}
It is easily seen that \eqref{exA01} fulfills $(H_0)$-$(H_3)$ with $\Omega=\R^2$ and the trivial solution $\phi^\ast=0$. For $\lambda^\ast\neq 0$ the assumption (i) holds. Moreover, the limit sets of \eqref{exA01} are singletons given by the limit equations
\begin{align*}
	x_{t+1}&=
	\begin{pmatrix}
		\alpha^{-1} & 0\\
		\lambda & \alpha
	\end{pmatrix}x_t,\quad\quad
	x_{t+1}=
	\begin{pmatrix}
		\alpha & 0\\
		\lambda & \alpha^{-1}
	\end{pmatrix}x_t.
\end{align*}
Both are hyperbolic with a $1$-dimensional stable subspace and hence $\alpha(\lambda),\omega(\lambda)$ are admissible yielding (ii). On the other hand, the triangular structure of \eqref{exA01} allows to compute the general solution $\vphi_\lambda(\cdot;0,\xi)$ for arbitrary initial values $\xi\in\R^2$. The first component $\vphi_\lambda^1$ is
	\begin{equation}
		\vphi_\lambda^1(t;0,\xi)=\alpha^{\abs{t}}\xi_1
		\fall t\in\Z,\,\xi\in\R^2
		\label{gensolfor}
	\end{equation}
	and consequently $\vphi_\lambda^1(\cdot;0,\xi)\in\ell_0$. For the second component this yields
	$$
		\vphi_\lambda^2(t;0,\xi)
		=
		\alpha^{-\abs{t}}\xi_2
		+
		\lambda
		\begin{cases}
			\sum_{s=0}^{t-1}\frac{1}{\alpha^{t-s-1}}\alpha^s\xi_1, & t\geq 0,\\
			-\sum_{s=t}^{-1}\alpha^{t-s-1}\alpha^{-s}\xi_1, & t<0
		\end{cases}
	$$
	and we arrive at the asymptotic representation
	$$
		\vphi_\lambda^2(t;0,\xi)
		=
		\begin{cases}
			\alpha^{-t}\intoo{\xi_2-\frac{\lambda\alpha}{\alpha^2-1}\xi_1}+o(1),
			& t\to\infty,\\
			\alpha^t\intoo{\xi_2+\frac{\lambda\alpha}{\alpha^2-1}\xi_1}+o(1),
			& t\to-\infty.
		\end{cases}
	$$
	Thus, for $\lambda\neq 0$ the inclusion $\vphi_\lambda(\cdot;0,\xi)\in\ell_0$ holds if and only if $\xi_2=\frac{\lambda\alpha}{\alpha^2-1}\xi_1$ and $\xi_2=-\frac{\lambda\alpha}{\alpha^2-1}\xi_1$, i.e.\ $\xi=(0,0)$. In conclusion, $0$ is the unique homoclinic solution to \eqref{exA01} for $\lambda\neq 0$, while in case $\lambda=0$ the trivial solution $\phi^\ast=0$ is embedded into a $1$-parameter family of homoclinic solutions. This means for every $\lambda^\ast\neq 0$ we are in the situation of \tref{thmglobdeq}(c) shown in \fref{figcon} (left).

\begin{exam}[transcritical bifurcation]\label{exC}
	Let $\delta\in\R\setminus\set{0}$ and consider the nonlinear difference equation
	$$
		x_{t+1}
		=
		f_t(x_t,\lambda)
		:=
		\begin{pmatrix}
			b_t& 0\\
			\lambda&c_t
		\end{pmatrix}x_t
		+
		\delta
		\begin{pmatrix}
			0\\
			(x_t^1)^2
		\end{pmatrix}
	$$
	with general solution $\vphi_\lambda$. Again $(H_0)$-$(H_3)$ hold with $\phi^\ast=0$. Since \eqref{prolin} is satisfied, we confirm assumption (i). As autonomous limit systems one gets
	\begin{align*}
		x_{t+1}=
		\begin{pmatrix}
			\alpha^{-1} & 0\\
			\lambda & \alpha
		\end{pmatrix}x_t
		+
		\delta
		\begin{pmatrix}
			0\\ (x_t^1)^2
		\end{pmatrix},\quad\quad
		x_{t+1}=
		\begin{pmatrix}
			\alpha & 0\\
			\lambda & \alpha^{-1}
		\end{pmatrix}x_t
		+
		\delta
		\begin{pmatrix}
			0\\ (x_t^1)^2
		\end{pmatrix}
	\end{align*}
	respectively. The first component of their general solutions $\vphi_\lambda^-(\cdot;\tau,\xi)$ and $\vphi_\lambda^+(\cdot;\tau,\xi)$ is bounded on $\Z$, if and only if $\xi_1=0$ holds, and plugging this into the second equations shows that the only bounded solution to the limit equations is the trivial one. Hence, $\alpha(\lambda),\omega(\lambda)$ are admissible and \tref{thmglobdeq} applies for $\lambda^\ast\neq 0$. More detailed, while the first component of $\vphi_\lambda(\cdot;0,\xi)$ given by \eqref{gensolfor} is homoclinic, the second component fulfills
	$$
		\vphi_\lambda^2(t;0,\xi)
		=
		\begin{cases}
			\alpha^{-t}\intoo{\xi_2-\frac{\delta\alpha}{\alpha^3-1}\xi_1^2-\frac{\lambda\alpha}{\alpha^2-1}\xi_1}+o(1),
			& t\to\infty,\\
			\alpha^t\intoo{\xi_2+\frac{\delta\alpha^2}{\alpha^3-1}\xi_1^2+\frac{\lambda\alpha}{\alpha^2-1}\xi_1}+o(1),
			& t\to-\infty;
		\end{cases}
	$$
	in summary, we see that $\vphi_\lambda(\cdot;0,\xi)$ is homoclinic if and only if $\xi=0$ or
	\begin{align*}
		\xi_1=-2\frac{\alpha^2+\alpha+1}{\delta(\alpha+1)^2}\lambda,\quad\quad
		\xi_2=-2\frac{\alpha(\alpha^2+\alpha+1)}{\delta(\alpha+1)^4}\lambda^2.
	\end{align*}
	Hence, besides the zero solution we have a unique nontrivial entire solution passing through the initial point $\xi=(\xi_1,\xi_2)$ at time $t=0$ for $\lambda\neq 0$. We are again in the setting of \tref{thmglobdeq}(c) shown in \fref{figcon} (center).
\end{exam}

\begin{exam}[pitchfork bifurcation]
	Let us suppose $\delta\neq 0$ in the nonlinear difference equation
	\begin{equation}
		x_{t+1}
		=
		f_t(x_t,\lambda)
		:=
		\begin{pmatrix}
			b_t& 0\\
			\lambda&c_t
		\end{pmatrix}x_t
		+
		\delta
		\begin{pmatrix}
			0\\
			(x_t^1)^3
		\end{pmatrix}.
		\label{exD01}
	\end{equation}
	As above we observe that assumption (i) holds. Moreover, the limit equations
	\begin{align*}
		x_{t+1}=
		\begin{pmatrix}
			\alpha^{-1} & 0\\
			\lambda & \alpha
		\end{pmatrix}x_t
		+
		\delta
		\begin{pmatrix}
			0\\ (x_t^1)^3
		\end{pmatrix},\quad\quad
		x_{t+1}=
		\begin{pmatrix}
			\alpha & 0\\
			\lambda & \alpha^{-1}
		\end{pmatrix}x_t
		+
		\delta
		\begin{pmatrix}
			0\\ (x_t^1)^3
		\end{pmatrix}
	\end{align*}
	possess no nontrivial bounded entire solutions, which results as in \eref{exC}. Therefore, the admissible limit sets $\alpha(\lambda),\omega(\lambda)$ allow to apply \tref{thmglobdeq}. In order to get a more detailed picture, note that the first component of the general solution to \eqref{exD01} is given by \eqref{gensolfor} and the second component reads as
	$$
		\vphi_\lambda^2(t;0,\xi)
		=
		\begin{cases}
			\alpha^{-t}\intoo{\xi_2-\frac{\delta\alpha}{\alpha^4-1}\xi_1^3-\frac{\lambda\alpha}{\alpha^2-1}\xi_1}+o(1),
			& t\to\infty,\\
			\alpha^t\intoo{\xi_2+\frac{\delta\alpha^3}{\alpha^4-1}\xi_1^3+\frac{\lambda\alpha}{\alpha^2-1}\xi_1}+o(1).
			& t\to-\infty.
		\end{cases}
	$$
	This asymptotic representation shows us that $\vphi_\lambda(\cdot;0,\xi)\in\ell_0$ holds if and only if $\xi=0$ or $\xi_1^2=-2\frac{1}{\delta}\lambda$ and $\xi_2=-2\frac{\alpha}{\alpha^4-1}\frac{(\delta\alpha^2+4\lambda+\delta)}{\delta^2}\lambda^2$. Again, the assertion of \tref{thmglobdeq}(c) holds and \fref{figcon} (right) gives a description of the sets $\Gamma_-,\Gamma_+$.
\end{exam}
\subsection{Semilinear equations}
\label{sec51}
It is well-known that linear-inhomogeneous equations $x_{t+1}=A_tx_t+\lambda b_t$ with $1\not\in\Sigma(A)$ and $b\in\ell_0$ possess unique homoclinic solutions $$\phi_t^\ast=\lambda\sum_{s\in\Z}G(t,s+1)b_s$$ continuing the trivial one for parameters $\lambda\neq 0$, where $G$ is the Green's function defined in \eqref{nogreen}. As a natural generalization of this setting, we consider semilinear difference equations
\begin{equation}
	\tag{$S_\lambda$}
	\boxed{x_{t+1}=A_tx_t+F_t(x_t,\lambda)}
	\label{eqsemi}
\end{equation}
with a nonlinearity $F_t:\R^d\tm\Lambda\to\R^d$, $t\in\Z$, satisfying $(H_0)$-$(H_2)$. In particular, in order to guarantee the admissibility assumption (ii), let us suppose that the following assumptions hold for all $\lambda\in\Lambda$:
\begin{itemize}
	\item[$\eqref{sec51}_1$] \eqref{deqlin} has an ED on $\Z$ with constants $\alpha,K$

	\item[$\eqref{sec51}_2$] $D_1^jF_t(0,\lambda^\ast)\equiv 0$ on $\Z$ and $\lim_{t\to\pm\infty}D_1^jF_t(0,\lambda)=0$ for $j\in\set{0,1}$

	\item[$\eqref{sec51}_3$] There exist functions $F^\pm:\R^d\tm\Lambda\to\R^d$ such that $F^\pm(0,\lambda)=0$,
	$$
		\lim_{t\to\pm\infty}\sup_{x\in B}\abs{F_t(x,\lambda)-F^\pm(x,\lambda)}=0\quad\text{for all bounded }B\subseteq\R^d
	$$
	and $\lip F^\pm(\cdot,\lambda)<\tfrac{K}{1-\alpha}$.
\end{itemize}
Here it is $\Omega=\R^d$ (for simplicity) and $f_t(x,\lambda)=A_tx+F_t(x,\lambda)$. With the reference parameter $\lambda^\ast=0$, due to assumption $\eqref{sec51}_2$ one can choose $\phi^\ast=0$ as homoclinic solution to $(S_0)$. Now keep an arbitrary $\lambda\in\Lambda$ fixed:

\underline{ad (i)}: From $D_1f_t(0,\lambda)=A_t+B_t(\lambda)$ with $B_t(\lambda):=D_1F_t(0,\lambda)$ we first obtain $1\not\in\Sigma(0)$ by assumption $\eqref{sec51}_1$. Moreover, the limit relation in $\eqref{sec51}_2$ for the derivative ensures that $\sL_{A+B(\lambda)}$ is a compact perturbation of $\eL_A$ (cf.~proof of \pref{propF}(c)). Hence, also $\sL_{A+B(\lambda)}$ is a Fredholm operator with index $0$ and \lref{lempalmer} ensures that $1\not\in\Sigma^\pm(\lambda)$ holds, i.e.\ \eqref{thmglobdeq2} is fulfilled.

\underline{ad (ii)}: Thanks to $\eqref{sec51}_3$ the limit sets of \eqref{eqsemi} consist of the respective semilinear equations
\begin{align}
	x_{t+1}=A_tx_t+F^-(x_t,\lambda),\quad\quad
	x_{t+1}=A_tx_t+F^+(x_t,\lambda)
	\label{semideqlim}
\end{align}
having the trivial solution. In addition, \pref{propsemlin} guarantees that they are the unique bounded entire solutions to \eqref{semideqlim} and thus $\alpha(\lambda),\omega(\lambda)$ are admissible.
\subsection{Asymptotically periodic equations}
\label{sec52}
The ED assumptions (i) of \tref{thmglobdeq} simplify and are easier to verify, when we restrict to asymptotically periodic equations, which can have different forward and backward periods:

Beyond $(H_0)$-$(H_3)$ we assume there exist $p_-,p_+\in\N$ so that the following holds for all $\lambda\in\Lambda$:
\begin{itemize}
	\item[$\eqref{sec52}_1$] There exist functions $f_t^\pm=f_{t+p_\pm}^\pm:\Omega\tm\Lambda\to\R^d$ for all $t\in\Z$ such that
	$$
		\lim_{t\to\pm\infty}\sup_{x\in B}\abs{f_t(x,\lambda)-f_t^\pm(x,\lambda)}=0
		\fall\text{bounded }B\subset\Omega
	$$

	\item[$\eqref{sec52}_2$] $1\not\in\Sigma(\lambda^\ast)$, there exist $p^\pm$-periodic sequences $(A_t^\pm(\lambda))_{t\in\Z}$ such that
	\begin{itemize}
		\item $D_1f_t(\phi_t^\ast,\lambda^\ast)$, $A_t^\pm(\lambda)$ are invertible

		\item $\lim_{t\to\pm\infty}\abs{D_1f_t(\phi_t^\ast,\lambda)-A_t^\pm(\lambda)}=0$

		\item the period matrices $\Pi^\pm(\lambda):=\Phi_A^\pm(p_\pm,0)$ satisfy $\sigma(\Pi^\pm(\lambda))\cap\S^1=\emptyset$

		\item the stable subspaces in forward and backward time fulfill
		$$
			\dim\bigoplus_{\substack{\lambda\in\sigma(\Pi^-(\lambda))\\ \abs{\lambda}<1}}\Eig_\lambda\sigma(\Pi^-(\lambda))
			=
			\dim\bigoplus_{\substack{\lambda\in\sigma(\Pi^+(\lambda))\\ \abs{\lambda}<1}}\Eig_\lambda\sigma(\Pi^+(\lambda))
		$$
	\end{itemize}

	\item[$\eqref{sec52}_3$] the trivial one is the only bounded entire solution to the limit equations
		\begin{align*}
			x_{t+1}=f_t^+(x_t,\lambda),\quad\quad
			x_{t+1}=f_t^-(x_t,\lambda)
		\end{align*}
\end{itemize}
In order to verify that \tref{thmglobdeq} applies, we keep $\lambda\in\Lambda$ fixed.

\underline{ad (i)}: It results from \eref{exper} and $\eqref{sec52}_2$ that $1\not\in\Sigma^\pm(A^\pm(\lambda))$. On both halfaxes $\Z_0^+$ and $\Z_0^-$ the equation \eqref{var} is an $\ell_0$-perturbation of the respective limit equations
\begin{align*}
	x_{t+1}=A_t^-(\lambda)x_t,\quad\quad
	x_{t+1}=A_t^+(\lambda)x_t
\end{align*}
and therefore \cite[Cor.~3.26]{poetzsche:11a} implies $\Sigma^\pm(\lambda)=\Sigma^\pm(A^\pm(\lambda))$.

\underline{ad (ii)}: From assumption $\eqref{sec52}_1$ we obtain the finite limit sets
\begin{align*}
	\alpha(\lambda)&=\set{f_{\cdot+s}^-(\cdot,\lambda):\Z\tm\Omega\to\R^d\mid 0\leq s<p_-},\\
	\omega(\lambda)&=\set{f_{\cdot+s}^+(\cdot,\lambda):\Z\tm\Omega\to\R^d\mid 0\leq s<p_+},
\end{align*}
which consists of the $p_\pm$-periodic limit functions, and their time translates. Due to $\eqref{sec52}_3$ these limit equations, in turn, merely have the trivial one, as bounded entire solution. Thus, the limit sets $\alpha(\lambda),\omega(\lambda)$ are admissible.

As a concretization we arrive at:
\begin{exam}[perturbed Beverton-Holt equation]
	Let $p_-,p_+\in\N$ and $(a_t)_{t\in\Z}$ be a positive sequence such that there exit $p_+$- resp.\ $p_-$-periodic sequences $(a_t^+)_{t\in\Z}$, $(a_t^-)_{t\in\Z}$ in $\R$ with $\lim_{t\to\pm\infty}\abs{a_t-a_t^\pm}=0$. The nonlinear scalar difference equation
	\begin{equation}
		x_{t+1}=\frac{a_tx_t}{1+\abs{x_t}}+\lambda b_t
		\label{bhex}
	\end{equation}
	with right-hand side $f_t(x,\lambda):=\frac{a_tx}{1+\abs{x}}+\lambda b_t$ satisfies $(H_0)$-$(H_3)$, provided $(b_t)_{t\in\Z}$ is a real sequence in $\ell_0$ and $\phi_t^\ast\equiv 0$. For $\lambda^\ast=0$ the variational equation of \eqref{bhex} along $\phi^\ast$ becomes $x_{t+1}=a_tx_t$ and \cite[Ex.~2.6(4)]{poetzsche:14} guarantees the dichotomy spectrum $\intcc{\min\set{c_-,c_+},\max\set{c_-,c_+}}$, where
	\begin{align*}
		c_-:=\sqrt[p_-]{a_{p_--1}^-\cdots a_0^-},\quad\quad
		c_+:=\sqrt[p_+]{a_{p_+-1}^+\cdots a_0^+}.
	\end{align*}
	If $1<\min\set{c_-,c_+}$ or $\max\set{c_-,c_+}<1$, then the variational equation $(V_0)$ has an ED on $\Z$, while \eqref{var} possess EDs on halfaxes with $P_t^\pm(\lambda)=1$. In order to apply \tref{thmglobdeq} with $\lambda^\ast=0$ it remains to ensure admissible limit sets of \eqref{bhex}. Thereto, notice that the limit equations of \eqref{bhex} are
	\begin{align*}
		x_{t+1}&=\frac{a_t^+x_t}{1+\abs{x_t}}=:f_t^+(x_t),&
		x_{t+1}&=\frac{a_t^-x_t}{1+\abs{x_t}}=:f_t^-(x_t)
	\end{align*}
	and $\lip f_t^\pm=a_t^\pm$ holds for all $t\in\Z$. Hence, by \eref{exapa} the assumption $c_-,c_+\in[0,1)$ implies that $\alpha(\lambda),\omega(\lambda)$ are admissible for all $\lambda\in\R$.
\end{exam}
\section{Outlook}
Rather than working with difference equations, similar results can be obtained in continuous time for finite-dimensional nonautonomous differential equations. Indeed, both approaches are largely parallel: Heteroclinic solutions are characterized as solutions to a nonlinear equation between the ambient function spaces $\eC_0^1$ and $\eC_0$. This infinite-dimensional equation is solved using the abstract global implicit function Thms.~\ref{thmkiel} and \ref{thmevequoz}, whose assumptions in turn rely on Fredholm and properness criteria. Despite of these similarities, as difference one has to mention that the counterpart to the operator $\sG$ acts between different spaces and that the compactness conditions in \lref{lemcomp} required for properness have to be adjusted.

A further alternative is to deal with Carath{\'e}odory differential equations. Such problems naturally occur as pathwise realization of random differential equations or in control theory. Here, an ambient spatial setting consists of the spaces $\eW_0^{1,\infty}$ and $\eL_0^\infty$ of absolutely continuous resp.\ essentially bounded functions vanishing at $\pm\infty$. These sets replace $\eC_0^1$ resp.\ $\eC_0$ in our above studies. Corresponding compactness or properness conditions can be found in \cite[Thm.~11, Lemma~12(ii)]{rabier:04}.

In the end, our methods also apply to spaces $\eW^{1,p}$ and $\eL^p$, $p\in(1,\infty)$, in continuous time, or $\ell^p$ in discrete time. This requires ambient growth conditions on the right-hand side of \eqref{deq} for the sake of well-defined substitution operators. Yet, such conditions might lack a physical motivation.
\appendix
\renewcommand{\theequation}{\Alph{section}.\arabic{equation}}
\renewcommand{\thesection}{\Alph{section}}
%
%
%
%
\section{Global continuation}
Let $X,Y$ denote Banach spaces. Global implicit function theorems describe the branch of zeros for a continuous mapping $G:\overline{O}\tm\Lambda\to Y$ containing a pair $(x^\ast,\lambda^\ast)\in X\tm\R$ such that
\begin{equation}
	G(x^\ast,\lambda^\ast)=0,
	\label{glob1}
\end{equation}
where $O\subseteq X$ is an open nonempty subset of $X$ with $x^\ast\in O$ and $\Lambda\neq\emptyset$ denotes an open interval containing $\lambda^\ast$. Throughout, suppose that the derivative $D_1G:O\tm\Lambda\to L(X,Y)$ exists as a continuous function satisfying
\begin{equation}
	D_1G(x^\ast,\lambda^\ast)\in GL(X,Y).
	\label{glob2}
\end{equation}
Therefore, the (local) implicit function theorem (cf.\ \cite[p.~7, Thm.~I.1.1]{kielhoefer:12}) applies and yields a local $C^0$-solution branch $\lambda\mapsto x(\lambda)$ to
$$
	G(x,\lambda)=0.
$$
We define $C\subseteq\overline{O}\tm\Lambda$ as maximal connected component of $G^{-1}(\set{0})\cap(\overline{O}\tm\Lambda)$ containing the local solution branch through $(x^\ast,\lambda^\ast)$. In order to obtain information on the global structure of $C$, two further assumptions are due. First, suppose Fredholmness
\begin{equation}
	D_1G(x,\lambda)\in F_0(X,Y)
	\fall(x,\lambda)\in O\tm\Lambda
	\label{glob3}
\end{equation}
and second, we require
\begin{equation}
	\begin{split}
	G|_{B\tm\Lambda_0}&
	\text{ is proper on closed, bounded}\\
	&B\subseteq O\text{ and compact $\Lambda_0\subseteq\Lambda$}.
	\end{split}
	\label{glob4}
\end{equation}
For globally defined $G$ one establishes
\begin{thm}[global implicit function theorem]\label{thmkiel}
	If \eqref{glob1}--\eqref{glob4} hold with $O=X$, $\Lambda=\R$, then exactly one of the following alternatives applies:
	\begin{enumerate}
		\item $C=\set{(x^\ast,\lambda^\ast)}\cup\Gamma_+\cup\Gamma_-$ with unbounded disjoint sets $\Gamma_-,\Gamma_+$

		\item $C\setminus\set{(x^\ast,\lambda^\ast)}$ is connected
	\end{enumerate}
\end{thm}
\begin{proof}
	The proof follows \cite[pp.~231--232, Thm.~II.6.1]{kielhoefer:12}, using the mod~$2$ reduction of the degree for proper $C^1$-Fredholm mappings of index zero, constructed by Fitzpatrick, Pejsachowicz and Rabier \cite{fitz-pejs-rabier:92, pejs-rabier:98}.
\end{proof}
\begin{SCfigure}[2]
	\includegraphics[scale=0.5]{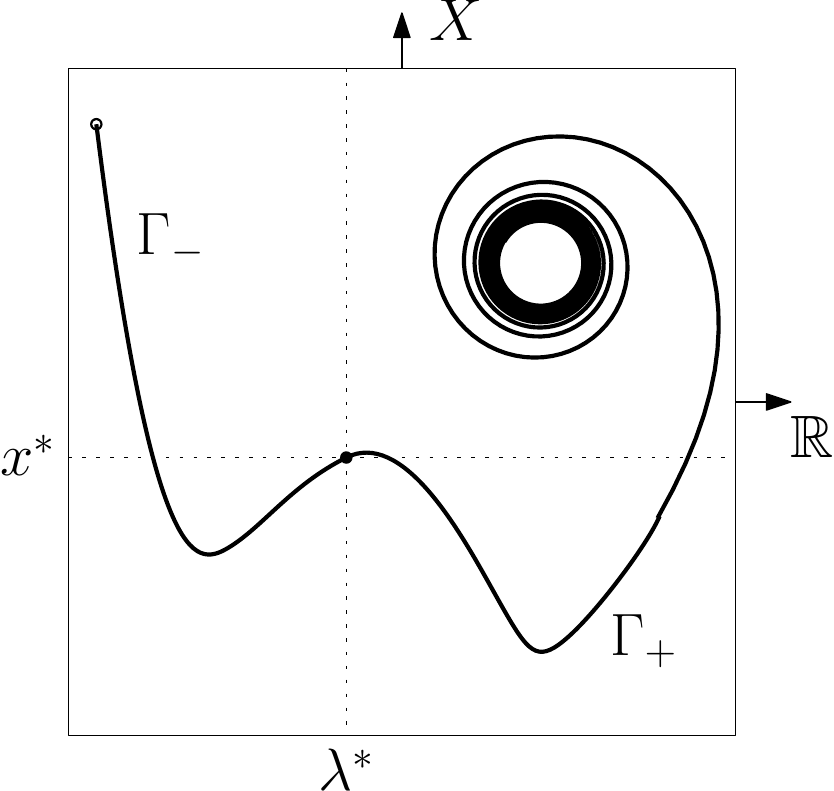}
	\caption{Situations ruled out by \tref{thmkiel}: The set $\Gamma_-$ is a curve having a finite limit as $\lambda\to-\infty$, while the other branch $\Gamma_+$ is bounded}
	\label{figglob2}
\end{SCfigure}
Note that \tref{thmkiel} rules out a situation as depicted in \fref{figglob2}. A variant of \tref{thmkiel} for ``local'' parameter spaces allows solution branches to end at the boundary of $O$ or $\Lambda$ and reads as
\begin{thm}[Ev{\'e}quoz's implicit function theorem]\label{thmevequoz}
	If \eqref{glob1}--\eqref{glob4} hold and
\begin{align*}
C_-:=\set{(x,\lambda)\in C\mid \lambda\leq\lambda^\ast},\;\; C_+:=\set{(x,\lambda)\in C\mid \lambda^\ast\leq\lambda},
\end{align*}
then at least one of the subsequent alternatives applies:
	\begin{enumerate}
		\item $C_-\cap C_+\neq\set{(x^\ast,\lambda^\ast)}$
		
		\item the branches $C_+$ and $C_-$ are connected and
		\begin{enumerate}
			\item[$(b_1)$] $C_+$ is unbounded or at least one of the following sets is nonempty:
			$$
				\Pi_1(\overline{C_+})\cap\partial O,\,
				\overline{\Pi_2(C_+)}\cap\partial\Lambda
			$$

			\item[$(b_2)$] $C_-$ is unbounded or at least one of the following sets is nonempty:
			$$
				\Pi_1(\overline{C_-})\cap\partial O,\,
				\overline{\Pi_2(C_-)}\cap\partial\Lambda,
			$$
		\end{enumerate}
		where $\Pi_1:X\tm\Lambda\to X$, $\Pi_2:X\tm\Lambda\to\Lambda$ stand for the projection of $(x,\lambda)$ onto the first resp.\ second component.
	\end{enumerate}
\end{thm}
\begin{proof}
	In \cite[Thm.~2.2]{evequoz:09} it is shown that
	\begin{itemize}
		\item[(a')] $C_+\cap C_-=\set{(x^\ast,\lambda^\ast)}$
	\end{itemize}
	yields (b). Since this implication $(a')\Rightarrow(b)$ is equivalent to $\neg(a')\vee(b)$ we obtain the assertion.
\end{proof}
\section{Topological dynamics}
\label{secB}
This appendix collects some required preliminaries from topological dynamics (cf.~\cite{sell:71,berger:siegmund:03}) and particular properties of the Bebutov flow.

Let $\Omega\subseteq\R^d$ be open. Given a continuous function $f\colon\Z\tm\Omega\to\R^d$ we define its \emph{hull} by
\begin{equation*}
	\fH(f):=\overline{\set{f(\cdot+s,\cdot):\Z\tm\Omega\to\R^d|\,s\in\Z}}\subseteq C(\Z\tm\Omega,\R^d).
\end{equation*}
This allows to introduce the \emph{Bebutov flow}
\begin{align}\label{H(f-g)}
	\sS^s\colon\fH(f)\to\fH(f),\quad\quad
	\sS^sg:=g(\cdot+s,\cdot)\fall s\in\Z
\end{align}
induced by $f$. The closure in the above definition of $\fH(f)$ is chosen w.r.t.\ an ambient topology such that $(s,g)\mapsto\sS^sg$ becomes continuous (cf.\ \cite{berger:siegmund:03}). Thus, \eqref{H(f-g)} defines a dynamical system on $\fH(f)$. Given a compact subset $K\subset\R^d$, it is convenient to write $K_\Omega:=K\cap\Omega$ and to denote $f$ as
\begin{itemize}
	\item \emph{bounded} on $K$, if $f(\Z\tm K_\Omega)\subseteq\R^d$ is bounded

	\item \emph{uniformly continuous} on $K$, if for every $\eps>0$ there is a $\delta>0$ with
	\begin{equation}\label{uniformly}
		\abs{x-y}<\delta
		\quad\Rightarrow\quad
		\sup_{t\in\Z}\abs{f(t,x)-f(t,y)}<\eps
		\fall x,y\in K_\Omega.
	\end{equation}
\end{itemize}

For instance, if $(t,x)\mapsto g(t,x)$ is bounded on bounded sets (uniformly in $t\in\Z$), then
$$
	\abs{g}_l:=\sup_{(t,x)\in\Z\tm(\bar B_l(0)\cap\Omega)}\abs{g(t,x)}
$$
are semi-norms yielding the \emph{compact-open topology}, i.e.\ the topology of uniform convergence on compact sets induced by the metric
\begin{equation}
	\bar d(g,\bar g):=\sum_{l=1}^\infty\frac{1}{2^l}\abs{g-\bar g}_l.
	\label{noddef}
\end{equation}
This construction of the Bebutov flow equips us with tools from dynamical systems in a natural way. For instance,
$$
	\omega(f)
	:=
	\set{g\in\fH(f)|\,\exists s_n\to\infty:\lim_{n\to\infty}\bar d(f(\cdot+s_n,\cdot),g)=0}
$$
defines the $\omega$-\emph{limit set} of $f$ and the $\alpha$-\emph{limit set} is
$$
	\alpha(f)
	:=
	\set{g\in\fH(f)|\,\exists s_n\to\infty:\lim_{n\to\infty}\bar d(f(\cdot-s_n,\cdot),g)=0}.
$$

\begin{lem}\label{lemproplim}
	If $f$ is bounded and uniformly continuous on any compact subset of $\Omega$, then $\fH(f)\neq\emptyset$ is compact and the following holds
	\begin{enumerate}
		\item $\alpha(g),\omega(g)\neq\emptyset$ are compact for all $g\in\fH(f)$

		\item the elements of $\alpha(f)$ and $\omega(f)$ are bounded and uniformly continuous on any compact subset of $\Omega$.
	\end{enumerate}
\end{lem}
\begin{proof}
	Due to \cite[Thm.~2.7, Rem.~2.8(ii)]{berger:siegmund:03} the hull $\fH(f)\neq\emptyset$ is compact.
	
	(a): Since the Bebutov flow is continuous (see also \cite[Thm.~2.7 and Rem.~2.8(ii)]{berger:siegmund:03}), the assertion is standard (see e.g.\ \cite[p.~11]{kloeden:rasmussen:10}).

	(b): Let $K\subset\R^d$ be compact and $g\in\omega(f)$. Hence, there exists a sequence $(s_n)_{n\in\N}$ in $\Z$ with $\lim_{n\to\infty}s_n=\infty$ such that
	$$
		\lim_{n\to\infty}\bar d(f_n,g)=0,
		\quad
		\text{ where }f_n(t,x):=f_n(t+s_n,x).
	$$
	Boundedness of $g(\Z\tm K_\Omega)$ readily follows from the corresponding property of the image $f(\Z\tm K_\Omega)$. In order to show that $g$ is uniformly continuous on $K$, we choose $\eps>0$. First, $g\in\omega(f)$ in the compact open topology guarantees that there exists a $N\in\N$ with
	\begin{align*}
		\abs{g(t,x)-f_n(t,x)}<\tfrac{\eps}{3},\quad\quad
		\abs{f_n(t,y)-g(t,y)}<\tfrac{\eps}{3}\fall n\geq N,\,t\in\Z
	\end{align*}
	and $x,y\in K_\Omega$. Second, by \eqref{uniformly} there is a $\delta>0$ such that $\abs{x-y}<\delta$ implies $\abs{f(t,x)-f(t,y)}<\tfrac{\eps}{3}$ for all $t\in\Z$ and $x,y\in K_\Omega$. Combining this with the triangle inequality and $n\geq N$ leads to	
	\begin{align*}
		&\abs{g(t,x)-g(t,y)}
		\leq
		\abs{g(t,x)-f_n(t,x)}+\abs{f_n(t,x)-f_n(t,y)}+\abs{f_n(t,y)-g(t,y)}<\\
		&<\tfrac{\eps}{3}+\tfrac{\eps}{3}+\tfrac{\eps}{3}=\eps\fall t\in\Z,\,x,y\in K_\Omega
	\end{align*}
	such that $\abs{x-y}<\delta$. Passing to the supremum over $t\in\Z$ implies \eqref{uniformly}, i.e.\ $g$ is uniformly continuous on $K$. The proof for $g\in\alpha(f)$ follows analogously, when $s_n$ is replaced by $-s_n$.
\end{proof}
\begin{exam}
	Almost periodic and almost automorphic functions $f$ yield a compact hull $\fH(f)$ (see \cite[Prop.~3.9]{berger:siegmund:03}) and thus compact limit sets.
\end{exam}
\begin{exam}[asymptotically periodic equations]\label{exap}
	A function $f$ as above is called \emph{asymptotically periodic}, if there exist $p_+,p_-\in\N$ and limit functions $f^\pm:\Z\tm\Omega\to\R^d$ satisfying $f^\pm(t,x)=f^\pm(t+p_\pm,x)$ and
	$$
		\lim_{t\to\pm\infty}\sup_{x\in B}\abs{f(t,x)-f^\pm(t,x)}=0\fall B\subseteq\Omega\text{ bounded}.
	$$
	This implies finite limit sets
	\begin{align*}
		\omega(f)&=\set{\sS^tf^+:\Z\tm\Omega\to\R^d|\,0\leq t<p_+},\\
		\alpha(f)&=\set{\sS^tf^-:\Z\tm\Omega\to\R^d|\,0\leq t<p_-}
	\end{align*}
	with $p_+$ resp.\ $p_-$ elements.
\end{exam}

\begin{lem}\label{lem225} 
	If $f$ is bounded and uniformly continuous on any compact subset of $\R^d$, then every sequence $(s_n)_{n\in\N}$ in $\Z$ with $\lim_{n\to\infty}\abs{s_n}=\infty$ has a subsequence $(s_{n_k})_{k\in\N}$ such that $(\sS^{s_k}f)_{k\in\N}$ converges.
\end{lem}
\begin{proof}
	Let us suppose w.l.o.g.\ that $\abs{s_n}\geq 1$ holds for all $n\in\N$, define the sets $C_n:=\Z\tm(\Omega\cap\bar B_n(0))$ and the restrictions
	\begin{align*}
		f_n:C_1\to\R^d,\quad\quad
		f_n(t,x):=f(t+s_n,x)\fall n\in\N.
	\end{align*}
	First, the boundedness of $f$ on $\bar B_n(0)$ shows that $(f_n)_{n\in\N}$ is a bounded sequence. Second, by the uniform continuity of $f$ on $\bar B_1(0)$ we see from \eqref{H(f-g)} that for every $\eps>0$ there exists a $\delta>0$ such that
	\begin{align*}
		\abs{x-y}<\delta
		\Rightarrow
		\abs{f_n(t,x)-f_n(t,y)}
		=
		\abs{f(t+s_n,x)-f(t+s_n,y)}<\eps
	\end{align*}
	holds for all $n\in\N$ and $t\in\Z$, $x,y\in C_1$; thus, the set $\set{f_n}_{n\in\N}$ of functions defined on $\Z\tm C_1$ is equicontinuous. By the Arzel{\'a}-Ascoli theorem (see \cite[p.~85]{yosida:80}) there is a sub\-sequence $(f_{n_m^1})_{m\in\N}$ of $(f_n)_{n\in\N}$ with $\abs{s_{n_m^1}}>1$ having a continuous limit $g_1:\Z\tm C_1\to\R^d$.

	Iterating this construction, for every integer $k\geq 2$ one extracts a subsequence $(n_m^k)_{m\in\N}$ from $(n_m^{k-1})_{m\in\N}$ with $\abs{s_{n_m^k}}>k$ such that the sequence $(f_{n_m^k})_{m\in\N}$ of restrictions
	\begin{align*}
		f_{n_m^k}:\Z\tm C_k&\to\R^d,\quad\quad
		f_{n_m^k}(t,x):=f(t+s_{n_m^k},x)
	\end{align*}
	converges uniformly to a continuous function $g_k:\Z\tm C_k\to\R^d$. By construction, each $g_{k+1}(t,\cdot)$ is an extension of $g_k(t,\cdot)$ to the compact set $C_{k+1}$, since passing to a subsequence has no effect on the values in $C_k\subseteq C_{k+1}$. This allows us to define the continuous function
	\begin{align*}
		g:\Z\tm\Omega \to\R^d,\quad\quad
		g(t,x):=g_k(t,x)\quad\text{if }(t,x)\in C_k
	\end{align*}
	and we claim that $g$ is the limit of the diagonal sequence $(\sS^{s_{n_m^m}}f)_{m\in\N}$. Indeed, for every compact $C\subseteq\R^d$ there exists a $n\in\N$ such that $C_\Omega\subseteq C_n$. Thus, $(\sS^{s_{n_m^k}}f)_{m\in\N}$ converges to $g$ uniformly on $\Z\tm C_\Omega$. Moreover, the remainder of the diagonal sequence $(\sS^{s_{n_m^m}}f)_{m\in\N}$ is a subsequence of $(\sS^{s_{n_m^k}}f)_{m\in\N}$ and converges uniformly to $g_k$ on $\Z\tm C_\Omega$. Since $g$ and $g_k$ have the same values on $\Z\tm C_\Omega$, this concludes our argument.
\end{proof}

A rather similar construction as in case of nonlinear functions $f$ is possible for bounded sequences $A\colon\Z\to L(\R^d)$: Indeed, one defines the \emph{hull}
\begin{equation*}
	\fH(A):=\overline{\set{A(\cdot+s):\Z\to L(\R^d)|\,s\in\Z}},
\end{equation*}
on which the \emph{Bebutov flow} reads as
\begin{align*}
	\sS^s:\fH(A)\to \fH(A),\quad\quad
	\sS^s(B):=B(\cdot+s)\fall s\in\Z.
\end{align*}
The closure in this definition of $\fH(A)$ is again taken in the uniform topology induced by the metric
$$
	\bar d(A,\bar A):=\sup_{t\in\Z}\abs{A(t)-\bar A(t)}
$$
and the limit sets now become
\begin{align*}
	\omega(A)
	&:=
	\set{B\in\fH(A)|\,\exists s_n\to\infty:\lim_{n\to\infty}\bar d(A(\cdot+s_n),B)=0},\\
	\alpha(A)
	&:=
	\set{B\in\fH(A)|\,\exists s_n\to\infty:\lim_{n\to\infty}\bar d(A(\cdot-s_n),B)=0}.
\end{align*}

\begin{lem}
If $A\colon\Z\to L(\R^d)$ is bounded, then $\fH(A)\neq\emptyset$ and the limit sets $\alpha(A)$, $\omega(A)$ are nonempty and compact.
\end{lem}
\begin{proof}
	The function $f:\Z\tm\R^d\to\R^d$, $f(t,x):=A(t)x$ is bounded and uniformly continuous on every set $\Z\tm K$ with a compact $K\subseteq\R^d$. Accordingly, \lref{lemproplim} applies and implies the claim.
\end{proof}
\section{Bounded solutions}
\label{appC}
In order to verify that a subset of the hull $\fH(f)$ is admissible and hence being able to apply \tref{thmglobdeq}, it is crucial to have criteria for the existence and uniqueness of bounded entire solutions at hand. For this purpose, let us consider nonautonomous difference equations \eqref{deq0} and begin with a folklore
\begin{lem}\label{susy}
	Let $X$ be a complete metric space. If a mapping $\sF:X\to X$ has a contractive iterate $\sF^p$, $p\in\N$, then $\sF$ possesses a unique fixed point.
\end{lem}
\begin{proof}
	Thanks to the contraction mapping principle, $\sF^p$ has a unique fixed point $x^\ast$. In order to see that $x^\ast$ is also a fixed point of $\sF$, we observe that any fixed point $y^\ast$ of $\sF$ satisfies $\sF^p(y^\ast)=y^\ast$ and thus $y^\ast=x^\ast$. Moreover, $\sF(x^\ast)=\sF(\sF^p(x^\ast))=\sF^p(\sF(x^\ast))$ guarantees that $\sF(x^\ast)$ is a fixed point of $\sF^p$ and consequently $x^\ast=\sF(x^\ast)$ by uniqueness.
\end{proof}

\begin{prop}[contractive equations]\label{propcontr}
	If $f_t:\R^d\to\R^d$, $t\in\Z$, are globally Lipschitz and satisfy
	\begin{itemize}
		\item[(i)] $f_t$ is bounded uniformly in $t\in\Z$, i.e.\ $\sup_{t\in\Z}\sup_{x\in B}\abs{f_t(x)}<\infty$ for all bounded $B\subset\R^d$,

		\item[(ii)] there exists a $n\in\N$ with
		\begin{equation}
			\sup_{t\in\Z}\prod_{s=t}^{t+n-1}\lip f_s<1,
			\label{exapa0}
		\end{equation}
	\end{itemize}
	then \eqref{deq0} has a unique bounded entire solution.
\end{prop}
\begin{rem}[expansive equations]
	The same conclusion as in \pref{propcontr} holds for expansive difference equations \eqref{deq0}. Here, $f_t:\R^d\to\R^d$, $t\in\Z$, are assumed to be bijective with Lipschitzian inverses satisfying conditions corresponding to (i) and (ii).
\end{rem}
\begin{proof}
	Notice that $\phi=(\phi_t)_{t\in\Z}\in\ell^\infty$ is an entire solution to \eqref{deq0}, if and only if $\phi$ is a fixed point of the mapping $\sF:\ell^\infty\to\ell^\infty$, $\sF(\phi)_t:=f_{t-1}(\phi_{t-1})$ for all $t\in\Z$, which is well-defined due to (i). Using mathematical induction it is not difficult to show that the iterates of $\sF$ allow the representation
	$$
		\sF^n(\phi)_t=f_{t-1}\circ\ldots\circ f_{t-n}(\phi_{t-n})\fall t\in\Z,\,\phi\in\ell^\infty,
	$$
	which guarantees
	\begin{align*}
		\abs{\sF^n(\phi)_t-\sF^n(\bar\phi)_t}
		&\leq
		\intoo{\prod_{s=t-n}^{t-1}\lip f_s}\abs{\phi_{t-n}-\bar\phi_{t-n}}\\
		&\leq
		\sup_{t\in\Z}\intoo{\prod_{s=t}^{t+n-1}\lip f_s}\norm{\phi-\bar\phi}
		\fall t\in\Z,\,n\in\N.
	\end{align*}
	This leads us to the Lipschitz estimate
	$$
		\norm{\sF^n(\phi)-\sF^n(\bar\phi)}
		\leq
		\sup_{t\in\Z}\intoo{\prod_{s=t}^{t+n-1}\lip f_s}\norm{\phi-\bar\phi}
		\fall\phi,\bar\phi\in\ell^\infty.
	$$
	Thus, $\sF^{p}$ is a contraction by \eqref{exapa0} and \lref{susy} with $X=\ell^\infty$ implies a unique fixed point $\phi$, which in turn is a bounded entire solution to \eqref{deq0}.
\end{proof}
\begin{exam}[asymptotically periodic equations]\label{exapa}
	We return to \eref{exap} and its terminology. If the $p_\pm$-periodic limit functions $f_t^\pm:\R^d\to\R^d$, $t\in\Z$, are globally Lipschitz with
	\begin{align*}
		\prod_{t=0}^{p_\pm-1}\lip f_t^\pm&<1,\quad\quad
		f_t^\pm(0)\equiv 0\on\Z,
	\end{align*}
	then the limit sets $\alpha(f),\omega(f)$ are admissible. Indeed, \pref{propcontr} implies unique bounded solutions $\phi^+,\phi^-$ to the respective limit equations
	\begin{align*}
		x_{t+1}=f_t^+(x_t),\quad\quad
		x_{t+1}=f_t^-(x_t)
	\end{align*}
	and finally, by uniqueness, $\phi^\pm=0$. So the limit sets are admissible.
\end{exam}

The following criteria address semilinear equations \eqref{deq0}, where
\begin{equation}
	f_t(x):=A_tx+r_t(x)
	\label{rhssem}
\end{equation}
with $A_t\in L(\R^d)$ and $r_t:\R^d\to\R^d$, $t\in\Z$. They require the \emph{Green's function}
\begin{equation}
	G(t,s)
	:=
	\begin{cases}
		\Phi_A(t,s)P_s,&s\leq t,\\
		-\bar\Phi_A(t,s)[I_{\R^d}-P_s],&s>t,
	\end{cases}
	\label{nogreen}
\end{equation}
where $P_t\in L(\R^d)$, $t\in\Z$, is an invariant projector for \eqref{deqlin}.
\begin{prop}[semilinear equations]\label{propsemlin}
	If $f_t:\R^d\to\R^d$, $t\in\Z$, are of the form \eqref{rhssem} with globally Lipschitzian $r_t:\R^d\to\R^d$, $t\in\Z$, satisfying
	\begin{itemize}
		\item[(i)] $1\not\in\Sigma(A)$

		\item[(ii)] $\sup_{t\in\Z}\lip r_t<\frac{K}{1-\alpha}$ (with the constants $K,\alpha$ from \eqref{edest}),
	\end{itemize}
	then \eqref{deq0} has a unique bounded entire solution.
\end{prop}
\begin{proof}
	We just sketch the argument and point out that the entire solutions $\phi\in\ell^\infty$ to \eqref{deq} can be characterized as solutions to the equation
	$$
		\phi_t=\sum_{s\in\Z}G(t,s+1)r_s(\phi_s)\fall t\in\Z.
	$$
	Thanks to the dichotomy estimates \eqref{edest} and assumption (i), a contraction mapping argument applies, provided the inequality (ii) holds.
\end{proof}
For our final criterion for the uniqueness of bounded entire solutions we introduce spaces of summable sequences depending on some $p\in[1,\infty)$:
\begin{align*}
	\ell^p(\R):=\set{(\phi_t)_{t\in\Z}:\,\sum_{t\in\Z}\abs{\phi_t}^p<\infty},\quad\quad
	\norm{\phi}_p:=\intoo{\sum_{t\in\Z}\abs{\phi_t}^p}^{1/p}
\end{align*}
\begin{prop}[asymptotically linear equations]
	Let $p,q\in(1,\infty)$ fulfill $\tfrac{1}{p}+\tfrac{1}{q}=1$. If $f_t:\R^d\to\R^d$, $t\in\Z$, are the form \eqref{rhssem} with
	\begin{itemize}
		\item[(i)] $\kappa:=\sup_{t\in\Z}\intoo{\sum_{s\in\Z}\abs{G(t,s+1)}^p}^{1/p}<\infty$

		\item[(ii)] there are $\rho,\mu\in\ell^q(\R)$ with $\abs{r_t(x)}\leq\mu_t+\rho_t\abs{x}$ for all $t\in\Z$, $x\in\R^d$

		\item[(iii)] there exist $R>0$ and $\lambda\in\ell^q(\R)$ with
		$$
			\abs{r_t(x)-r_t(\bar x)}\leq\lambda_t\abs{x-\bar x}\fall t\in\Z,\,x,\bar x\in B_R(0)
		$$

		\item[(iv)] $\norm{\rho}_q+\norm{\mu}_q<\tfrac{1}{2\kappa}$ and $\norm{\lambda}_q<\tfrac{1}{\kappa}$,
	\end{itemize}
	then \eqref{deq0} has a unique bounded entire solution.
\end{prop}
In case $1\not\in\Sigma(A)$ the assumption (i) holds with $\kappa:=K\alpha\intoo{\frac{1+\alpha^p}{1-\alpha^p}}^{1/p}$.
\begin{proof}
	See \cite[Thm.~3.2]{fenner:pinto:97}.
\end{proof}

\section*{Acknowledgement}
\noindent
The second author is supported by the NCN Grant 2013/09/B/ST1/01963.

\end{document}